\documentclass[12pt]{article}
\usepackage{float}
\usepackage[latin9]{inputenc}
\usepackage{subfigure}
\usepackage{geometry}

\usepackage[dvipdfmx]{graphicx}
\usepackage{algorithm}
\usepackage{algpseudocode}
\usepackage{tikz}
\RequirePackage{amsthm,amsmath,amssymb}
\usepackage{mathrsfs}
\usepackage{authblk}
\newcommand{\doBlank}[1]{}
{}

\def\?#1{}

\usepackage{color}

\newcommand{\mcf}{\mathcal{F}}
\newcommand{\mcl}{\mathcal{L}}
\newcommand{\mcc}{\mathcal{C}}
\newcommand{\mbbh}{\mathbb{H}}

\newcommand{\mbbr}{\mathbb{R}}
\newcommand{\mbbl}{\mathbb{L}}
\newcommand{\mbbn}{\mathbb{N}}
\newcommand{\mbbu}{\mathbb{U}}
\newcommand{\mbby}{\mathbb{Y}}
\newcommand{\mbbz}{\mathbb{Z}}
\newcommand{\al}{\alpha}
\newcommand{\sig}{\sigma}
\newcommand{\gam}{\gamma}
\newcommand{\p}{\partial}
\newcommand{\ep}{\epsilon}
\newcommand{\del}{\delta}

\newcommand{\lam}{\lambda}
\newcommand{\cip}{\xrightarrow{p}} % <- Convergence in probability

\def\ds#1{\displaystyle{#1}}
\def\nn{\nonumber}
\def\lp{L\'{e}vy process}
\def\tz{\theta_{0}}
\def\tes{\hat{\theta}_{N}}
\def\aes{\hat{\al}_{N}}
\def\ges{\hat{\gam}_{N}}
\def\pr{\mathbb{P}}
\def\E{\mathbb{E}}

\DeclareMathOperator{\diag}{diag}

\allowdisplaybreaks

\newcommand{\dif}{\mathrm{d}}

\numberwithin{equation}{section}
\theoremstyle{plain}
\newtheorem{thm}{Theorem}[section]
\newtheorem{assumption}{Assumption}
\newtheorem{lemma}{Lemma}[section]

\newtheorem{definition}{Definition}
\newtheorem{proposition}{Proposition}[section]

\newtheorem{theorem}{Theorem}[section]

\newcommand{\writetitle}{0}
\newcommand{\mytitle}[1]
{   \ifthenelse{\writetitle=1}{}{}
}

\newread\mysource
\begin{document}
\title{Bayesian inference for Stable L\'{e}vy driven Stochastic Differential Equations with high-frequency data}
\date{This version: \today}
\author[1]{Ajay Jasra\thanks{staja@nus.edu.sg}}
\author[2]{Kengo Kamatani\thanks{kamatani@sigmath.es.osaka-u.ac.jp}}
\author[3]{Hiroki Masuda\thanks{hiroki@math.kyushu-u.ac.jp}}
\affil[1]{Department of Statistics  Applied Probability, National University of Singapore, SG}
\affil[2]{Department of Engineering Science, Osaka University, JP}
\affil[3]{Faculty of Mathematics, Kyushu University, JP}

%\tableofcontents
%\newpage
\maketitle
\begin{abstract} 
In this article we consider parametric Bayesian inference for stochastic differential equations (SDE) driven by a pure-jump stable L\'{e}vy process, which is observed at high frequency. In most cases of practical interest, the likelihood function is not available, so we use a quasi-likelihood and place an associated prior on the unknown parameters. It is shown under regularity conditions that  there is a Bernstein-von Mises theorem associated to the posterior. We then develop a Markov chain Monte Carlo (MCMC) algorithm for Bayesian inference and assisted by our theoretical results, we show how to scale Metropolis-Hastings proposals when the frequency of the data grows, in order to prevent the acceptance ratio going to zero in the large data limit. Our algorithm is presented on numerical examples that help to verify our theoretical findings.
\end{abstract}

\section{Introduction}

Stochastic differential equations (SDE) are found in a wide variety of real applications, such as finance and econometrics (see for instance \cite{cont} and the references therein), mathematical biology, movement ecology, turbulence, signal processing, to mention just a few. In this article we are concerned with Bayesian parameter estimation of discretely observed SDE driven by a pure-jump stable L\'{e}vy process with high-frequency data, which is often found in the afore-mentioned applications.
Among others, we refer to \cite{JW94}, \cite{SamTaq}, and \cite{MR1739520} for a comprehensive account of stable distribution and stochastic processes driven by a stable {\lp}.

Often, the main challenge with Bayesian or classical inference with L\'{e}vy driven SDE is the lack of tractability of the transition density of the process, hence one cannot evaluate a non-negative and unbiased estimate of the transition density. The latter is often required for inference methods.
In such scenarios, one often has to resort to time-discretization. We focus on the time discretization induced by the quasi-likelihood
approach, which, under suitable regularity conditions, enables us to deduce several desirable properties, such as consistency and asymptotic (mixed) normality; see \cite{masuda13,masuda16.sqmle} and the references therein.
As a result, one expects rather favorable properties of the associated posterior distribution in similar settings. In addition,
one must construct inference techniques, in this article MCMC, which will turn out to be robust in the high-frequency limit. That is, algorithms that will not collapse in some sense.

Bayesian inference for diffusions and jump diffusions have been considered in many articles.
This includes the fully observed (jump) diffusion case \cite{golight,roberts} and the partially observed jump diffusion 
case \cite{bns,gander,griffin,jasra}. In particular the work \cite{peters} considered some related, but different classes of models to this article,
except using approximate Bayesian computation \cite{marin} methods with MCMC. 
We note that most of the previous works do not consider quasi-likelihood and/or the large data limit of the performance of MCMC.
Our contributions of this article are roughly summarized as follows:
\begin{itemize}
\item{To construct an approximate posterior distribution with favorable theoretical properties.
That is, under assumptions, that there is a Bernstein-von Mises theorem associated to the posterior.}
\item{To develop an MCMC algorithm and assisted by our theoretical results, show how to scale
Metropolis-Hastings proposals when the frequency of the data grows, in order to prevent the acceptance ratio going to zero in the large data limit.}
\end{itemize}

%\begin{figure}\centering
%\label{fig:data}
%\includegraphics[width=.9\textwidth,height=6cm]{histogram.pdf}
%%\includegraphics[width=14cm,bb= 0 0 1004 748]{histogram.pdf}
%\caption{Histogram of the daily log return of IBM stock from 2014/04/01 to 2014/04/30.
%The estimated value of $\beta$ was $1.41$ }
%\label{fig:histogram}
%\end{figure}

This article is structured as follows. In Section \ref{sec:model}, the model setup and an associated Bernstein-von Mises theorem are given.
Our MCMC algorithm is also described. In Section \ref{sec:theory},
our theoretical results for MCMC are stated. Section \ref{sec:simos} provides some numerical simulations which confirm our theoretical findings.
We also analyze a real data set from the NYSE %(see Figure \ref{fig:data}) 
which features properties often captured well by the processes we study in this article. 
The appendix features a variety of proofs for our technical results.

%%%%%
%%%%%
\section{Stable-L\'{e}vy SDEs, Model and Algorithm}\label{sec:model}

\subsection{Setup}\label{sec:setup}

Let $(\Omega,\mathcal{F},(\mcf_{t})_{t\in[0,T]},\mathbb{P})$ be a complete stochastic basis. 
Let $\{X_t;t\in [0,T]\}$ be a solution to the one-dimensional stochastic differential equation
\begin{align}
	\dif X_t=a(X_t,\alpha)\dif t+c(X_{t-},\gamma)\dif J_t,
\label{hm:SDE}
\end{align}
where $\{J_t\}_t$ is the stable L\'{e}vy-process independent of the initial variable $X_{0}$ and such that
\begin{equation}
\mathbb{E}(e^{iuJ_{1}})=e^{-|u|^{\beta}},
\nonumber
\end{equation}
where it is assumed throughout that $\mcf_{t}=\sig(J_{s}: s\le t)\vee\sig(X_{0})$ and that $1\le \beta <2$.
Denote by $\phi_{\beta}$ the $\beta$-stable density of $\mcl(J_{1})$.
We will write $\mathbb{P}_{\theta}$ for the image measure of $\mcl(X_{t})$ associated with the parameter value
\begin{equation}
\theta:=(\al,\gam)\in\Theta_{\al}\times\Theta_{\gam}=\Theta \subset\mbbr^{p},
\nonumber
\end{equation}
with $\Theta_{\al}\subset\mbbr^{p_{\al}}$ and $\Theta_{\gam}\subset\mbbr^{p_{\gam}}$ being bounded convex domains. The true value is denoted by $\tz=(\al_{0},\gam_{0})\in\Theta$. For brevity, we will also write $\pr$ for $\pr_{\theta_{0}}$ as well.
Let $\overline{\Theta}$ denote the closure of $\Theta$, and write $a\lesssim b$ if $a\le Cb$ for some universal constant $C>0$. 

\begin{assumption}[Regularity of the coefficients]$\ $
\begin{enumerate}
\item The functions $a(\cdot,\al_{0})$ and $c(\cdot,\gamma_{0})$ are globally Lipschitz and of class $\mcc^{2}(\mbbr)$, and $c(x,\gamma)$ is positive for every $(x,\gamma)$;
\item $a(x,\cdot)\in\mcc^{3}(\Theta_{\al})$ and $c(x,\cdot)\in\mcc^{3}(\Theta_{\gam})$ for each $x\in\mbbr$;
\item $\ds{\sup_{\theta\in\overline{\Theta}} \bigg\{
\max_{0\le k\le 3}\max_{0\le l\le 2}\bigg(
\left|\p_{\al}^{k}\p_{x}^{l}a(x,\al)\right| 
+ \left|\p_{\gam}^{k}\p_{x}^{l}c(x,\gam)\right|\bigg) + c^{-1}(x,\gam)\bigg\} \lesssim 1+ |x|^{C}}$.
\end{enumerate}
\label{hm:A1}
\end{assumption}

\begin{assumption}[Identifiability conditions]
The random functions $t\mapsto\left(a(X_{t},\al), c(X_{t},\gam)\right)$ and $t\mapsto\left(a(X_{t},\al_{0}), c(X_{t},\gam_{0})\right)$ 
on $[0,T]$ a.s. coincide if and only if $\theta=\tz$.
%\begin{equation}
%\mathbb{P}\left( \left(a(X_{t},\al), c(x,\gam)\right)_{t\in[0,T]} = \left(a(X_{t},\al_{0}), c(x,\gam_{0})\right)_{t\in[0,T]} \right)=1\quad
%\iff\quad \theta = \theta_{0}.
%\nonumber
%\end{equation}
\label{hm:A2}
\end{assumption}

The process $X$ is observed only at discrete time $t=nh\ (j=0,\ldots, N)$ where $h=T/N$ with the terminal sampling time $T>0$ being fixed.
Let $X^N:=\{X_{nh};n=0,\ldots, N\}$, the available data set.
The Euler-Maruyama discretization under $\mathbb{P}_{\theta}$ is then given by
\begin{equation}
\Delta^N_nX \approx a(X_{(n-1)h},\alpha)h+c(X_{(n-1)h},\gamma)\Delta^N_n J,
\label{hm:Euler}
\end{equation}
where $\Delta^N_n X:=X_{nh}-X_{(n-1)h}$ and $\Delta^N_nJ:=J_{nh}-J_{(n-1)h}$.
This suggests to consider the following stable quasi-likelihood \cite{masuda16.sqmle} (the multiplicative constant ``$h^{-N/\beta}$'' omitted):
\begin{align}
	\mathbb{L}_N(\theta|X^N)=\prod_{n=1}^N\frac{1}{c(X_{(n-1)h},\gamma)}\phi_\beta\left(\frac{\Delta^N_nX-a(X_{(n-1)h},\alpha)h}{c(X_{(n-1)h},\gamma)h^{1/\beta}}\right). 
\label{hm:sql_def}
\end{align}
We define the stable quasi-maximum likelihood estimator by any element
\begin{equation}
\tes=(\aes,\ges)\in \mathop{\rm argmax}_{\theta\in\overline{\Theta}} \log\mbbl_{N}(\theta|X^{N}).
\nonumber
\end{equation}
Note that $\mathbb{L}_N(\theta|X^N)$ is the true likelihood only when both $a(x,\al)$ and $c(x,\gam)$ are constants.

It follows from \cite{masuda16.sqmle} that under Assumptions \ref{hm:A1} and \ref{hm:A2} the estimator $\tes$ is asymptotically mixed-normally distributed: on a suitably extended probability space $(\tilde{\Omega},\tilde{\mathcal{F}},\tilde{\mathbb{P}})$ carrying a $p$-dimensional standard Gaussian random variable  $\eta$ independent of $\mcf$, we have
\begin{equation}
%\left(\sqrt{n}h^{1-1/\beta}(\aes-\al_{0}),\, \sqrt{n}(\ges-\gam_{0})\right) \overset{\mcl}\to I(\tz)^{-1/2}\eta,
\left(\sqrt{n}h^{1-1/\beta}(\aes-\al_{0}),\, \sqrt{n}(\ges-\gam_{0})\right) \overset{\mcl_{\mathrm{st}}}\to I(\tz)^{-1/2}\eta,
\label{hm:amn}
\end{equation}
where $\overset{\mcl_{\mathrm{st}}}\to$ stands for the $\mcf$-stable convergence in law (see \cite{JacPro11} for details)
and where $I(\tz)=\diag\left( I_{\al}(\tz),\, I_{\gam}(\gam_{0})\right)$ denotes the a.s. positive definite quasi Fisher-information matrix specified by
\begin{equation}
I_{\al}(\tz)=C_{\al}(\beta)\Sigma_{T,\al}(\theta_{0}), \qquad I_{\gam}(\gam_{0})= C_{\gam}(\beta)\Sigma_{T,\gam}(\gam_{0})
\nonumber
\end{equation}
with
\begin{align}
& C_{\al}(\beta) := \int \bigg(\frac{\partial\phi_{\beta}}{\phi_{\beta}}(y)\bigg)^{2}\phi_{\beta}(y)\dif y, \qquad 
C_{\gam}(\beta):=\int \bigg(1+ y\frac{\partial\phi_{\beta}}{\phi_{\beta}}(y)\bigg)^{2}
\phi_{\beta}(y)\dif y, \nn\\
\nn\\
& \Sigma_{T,\al}(\theta_{0}) :=\frac{1}{T}\int_{0}^{T}\frac{\{\partial_{\al}a(X_{t},\al_{0})\}^{\otimes 2}}{c^{2}(X_{t},\gam_{0})}\dif t, \qquad 
\Sigma_{T,\gam}(\gam_{0}) :=\frac{1}{T}\int_{0}^{T}\frac{\{\partial_{\gam}c(X_{t},\gam_{0})\}^{\otimes 2}}{c^{2}(X_{t},\gam_{0})}\dif t.
\nn%\label{av.def_gam}
\end{align}
It is known that the stable quasi-maximum likelihood estimator $\hat{\theta}_{N}$ attains the Haj\'{e}k-Le Cam minimax lower bound in some special cases, see \cite{CleGlo15} and \cite{Mas09}.
The convergence \eqref{hm:amn} shows that we do have the conventional Studentization (asymptotic standard normality) result without a finite-variance property as well as any stability condition such as ergodicity.

%%%%%
%%%%%
\subsection{Posterior and Markov chain Monte Carlo methods}\label{sec:mcmc}

\subsubsection{Quasi-posterior distribution}

%\subsection{Markov chain Monte Carlo for the extended space}

%%The simplest way is to perform the random-walk Metropolis-Hastings algorithm with the target distribution $\Pi(\theta|X^N)$. 
%We apply MCMC with the target distribution 
%$\Pi(\theta,V^N|X^N,V^N)$. 
%We work on the extended parameter $(\theta, V^N)$ where $V^N=(V_n;n=1,\ldots, N)\in\mathbb{R}_+^N$. 
%Let $(\theta, V^N)$ be an initial guess. We propose $(\theta^*, V^{*N})$ where $V^{*N}=(V_n^{N*})_n$ by 
%\begin{align*}
%	\theta^*\leftarrow \theta+(\mathrm{noise}),\ 
%	V_n^{N*}\sim (\mathrm{positive\ stable\ distribution\ with}\ \beta/2)
%\end{align*}
%Accept the proposal by 
%\begin{align*}
%	\min\left\{1,\frac{\mathbb{L}_N(\theta^*|X^N,V^{N*})\pi(\theta^*)}{\mathbb{L}_N(\theta|X^N,V^N)\pi(\theta)}\right\}.
%\end{align*}

%\subsection{Metropolis-within-Gibbs algorithm}

The quasi-posterior distribution is given by
\begin{align*}
	\Pi_N(\dif\theta|X^N)\propto \mathbb{L}_N(\theta|X^N)\Pi(\dif\theta),
\end{align*}
where $\Pi(\dif\theta)=\pi(\theta)\dif\theta$ denotes the prior distribution of $\theta$. The $\beta$-stable density $\phi_{\beta}$ does not have a closed form for $\beta\in(1,2)$, hence cannot be computed pointwise and neither then can the posterior up-to a constant; this is required for stochastic simulation algorithms such as MCMC. In general, $\phi_{\beta}$ can be only computed via certain numerical integration, the iteration of which may be rather time-consuming \cite{MatTak06}. However, we have access to a non-negative unbiased estimate of the un-normalized posterior, which does suffice to apply (e.g.) MCMC and it may be constructed as follows.
Let $F_\beta(\dif v)=f_\beta(v)\dif v$ denote the positive $\beta/2$-stable distribution with Laplace transform 
$$
\int_0^\infty e^{-vt}F_\beta(\dif v)=e^{-|2t|^{\beta/2}}\quad (t\ge 0). 
$$
Then, the stable distribution $\mcl(J_{1})$ is the law of $\sqrt{V}W$ where $W\sim N(0,1)$ and $V\sim F_\beta$, 
and we have the normal variance-mixture representation:
\begin{align}\label{eq:pb}
	\phi_\beta(x)=\int_0^\infty\frac{1}{\sqrt{2\pi v}}\exp\left(-\frac{x^2}{2v}\right)F_\beta(\dif v). 
\end{align}
If $V_1,\ldots, V_N \sim F_\beta$, 
and $\xi_1,\ldots, \xi_N\sim N(0,1)$, then $\mathcal{L}(h^{1/\beta}\sqrt{V_n}\xi_n)=\mathcal{L}(\Delta^N_nJ)$.
Invoking \eqref{hm:Euler} we may expect that the formal distributional approximation
\begin{align*}
    \Delta^N_nX &\approx a(X_{(n-1)h},\alpha)h+c(X_{(n-1)h},\gamma)h^{1/\beta}\sqrt{V_n}\xi_n
\end{align*}
under $\mathbb{P}_{\theta}$ would be meaningful.

Building on the above observation, we will make use of the ``complete'' quasi-likelihood 
\begin{align}\nn%\label{eq:qlik}
	\mathbb{L}_N(\theta|X^N,V^N)=\prod_{n=1}^N\frac{1}{c(X_{(n-1)h},\gamma)\sqrt{V_n}}\phi\left(\frac{\Delta^N_nX-a(X_{(n-1)h},\alpha)h}{c(X_{(n-1)h},\gamma)\sqrt{V_n}h^{1/\beta}}\right)f_\beta(V_n),
\end{align}
where $\phi$ denotes the standard normal density. The corresponding posterior distribution is
\begin{align*}
	\Pi_N(\dif\theta,V^N|X^N)\propto \mathbb{L}_N(\theta|X^N,V^N)\Pi(\dif\theta). 
\end{align*}
Note that we can still not compute $\Pi_N(\dif\theta,V^N|X^N)$ up to a constant. Nevertheless, we will still be able to devise a Metropolis-within-Gibbs MCMC algorithm which mitigates this issue.

\subsubsection{MCMC Algorithm}

Let $\theta\in\Theta$. For $n=1,\ldots, N$,  generate $V_n$ from $F_\beta(\dif v|\epsilon_n(\theta))$
where
\begin{align}\label{eq:fb}
	F_\beta(\dif v|x)=
	\left.\frac{1}{\sqrt{2\pi v}}\exp\left(-\frac{x^2}{2v}\right)F_\beta(\dif v)\right/\phi_\beta(x)
\end{align}
and 
\begin{align*}
	\epsilon_n(\theta)= \frac{\Delta^N_nX-a(X_{(n-1)h},\alpha)h}{c(X_{(n-1)h},\gamma)h^{1/\beta}}. 
\end{align*}
Then update $\theta^*\leftarrow \theta+(\mathrm{noise})$
and accept $\theta^*$ with probability 
\begin{align*}
	\alpha(\theta,\theta^*)=\min\left\{1,\frac{\mathbb{L}_N(\theta^*|X^N,V^N)\pi(\theta^*)}{\mathbb{L}_N(\theta|X^N,V^N)\pi(\theta)}\right\}.
\end{align*}
This noise should be properly scaled by the rate matrix
\begin{equation}
D_{N} := \mathrm{diag}\big( \sqrt{N}h^{1-1/\beta}I_{p_{\al}},\, \sqrt{N}I_{p_{\gam}} \big),
\nonumber
\end{equation}
 as specified in \eqref{hm:amn}. 
 The random variable $V_n\sim F_\beta(\dif v|\epsilon_n(\theta))$ can be sampled via rejection sampling for each $n\in\{1,\dots,N\}$:
Generate $v\sim F_\beta(\dif v)$ and accept it with probability
$$
\left.\left\{v^{-1/2}\exp\left(-\frac{x^2}{2v}\right)\right\} \right/ \left\{|x|^{-1/2}\exp\left(-\frac{|x|}{2}\right)\right\}. 
$$
The algorithm is presented in Algorithm \ref{alg:1}
(We write $N_{p}(\mu,\Sigma)$ for the $p$-dimensional Gaussian distribution with mean $\mu$ and covariance matrix $\Sigma$).

\begin{algorithm}
\caption{Metropolis-within-Gibbs algorithm for discretely observed stable L\'{e}vy process}	
\begin{algorithmic}
\State Initialize $\theta_0\in\mathbb{R}^p$, positive definite matrix $\Sigma\in\mathbb{R}^{p\times p}$. 
\For {$m  = 1,\dots, M-1$}
	\For {$n  = 1,\dots, N$}
		\State $V_n\gets F_\beta(\dif v|\epsilon_n(\theta_{m-1}))$
	\EndFor
	\State $W_m\sim N_p(0, \Sigma)$,  $U_m\sim \mathrm{Unif}[0,1]$
	\State $\theta^*_m\gets \theta_{m-1} + D_N^{-1}W_m$
    \If{$U_m\le \alpha(\theta_{m-1},\theta^*_m)$}
    \State $\theta_m\gets\theta_m^*$
    \Else
    \State $\theta_m\gets\theta_{m-1}$
    \EndIf
\EndFor
\end{algorithmic}
\label{alg:1}
\end{algorithm}

A minor modification is to consider a (correlated) pseudo-marginal algorithm \cite{2015arXiv151104992D} to avoid rejection sampling. 
In this case, update $\theta$ as in Algorithm in \ref{alg:1}
together with
$$
V_n^*\leftarrow \rho^{2/\beta} V_n + (1-\rho)^{2/\beta}\xi_n,\quad \xi_n\sim F_\beta,
$$
as in \cite{MR1723510} and \cite{2016arXiv160202889K}. 
Then accept $(\theta^*, V^{N*})$  where $V^{N*}=\{V_n^*\}_{n=1,\ldots N}$ with probability 
\begin{align*}
	\alpha((\theta,V^N),(\theta^*,V^{N*}))=\min\left\{1,\frac{\mathbb{L}_N(\theta^*|X^N,V^{N*})\pi(\theta^*)}{\mathbb{L}_N(\theta|X^N,V^N)\pi(\theta)}\right\}.
\end{align*}
Here, $\rho$ is a tuning parameter and we recommend to take $\rho$ close to $1$. 
%\tcr{We have found that rejection sampling proposing from the stable distribution works well if $\beta/2\in(0.5,1)$ and $h\approx 1$. Otherwise, the computational time may be too long and hence the latter method is preferred.}
%\begin{algorithm}
%\caption{Pseudo-marginal algorithm for discretely observed stable L\'{e}vy process}	
%\begin{algorithmic}
%\State Initialize $\theta_0\in\mathbb{R}^p, V^N_0\in \mathbb{R}_+^N$, positive definite matrix $\Sigma\in\mathbb{R}^{p\times p}$ and $\rho\in (0,1)$. 
%\For {$m  = 1,\dots, M-1$}
%	\For {$n  = 1,\dots, N$}
%		\State $V_{n,m}^*\leftarrow \rho^{2/\beta} V_{n,m-1} + (1-\rho)^{2/\beta}\xi_{n,m},\ \xi_{n,m}\sim F_\beta$
%	\EndFor
%	\State Set $V_m^{N*}=\{V_{n,m}^*\}_n$
%	\State $W_m\sim N_p(0, \Sigma)$,  $U_m\sim \mathrm{Unif}[0,1]$
%	\State $\theta^*_m\gets \theta_{m-1} + D_N^{-1}W_m$
%    \If{$U_m\le \alpha((\theta_{m-1},V^N_{m-1}),(\theta^*_m, V^{N*}_{m}))$}
%    \State $\theta_m\gets\theta_m^*$, $V_m^N\leftarrow V_m^{N*}$
%    \Else
%    \State $\theta_m\gets\theta_{m-1}$, $V_m^N\leftarrow V_{m-1}^N$
%    \EndIf
%\EndFor
%\end{algorithmic}
%\label{alg:2}
%\end{algorithm}

%%%
\subsection{Bernstein-von Mises theorem}
\label{hm:sec_BvM}

In this section we give a Bernstein-von Mises theorem associated with the stable quasi-likelihood \eqref{hm:sql_def}. Different from the classical version \cite{BKP71}, we will look at the convergence of the posterior distribution of the rescaled parameter centered not at $\tes$, but at $\tz$.

\begin{assumption}
The prior distribution $\Pi$ admits a bounded Lebesgue density $\pi$ which is continuous and positive at $\tz$.
\label{hm:A3}
\end{assumption}

Let 
\begin{equation}
\mbbh_{N}(\theta):=\log\mbbl_{N}(\theta|X^{N})
\nonumber
\end{equation}
be the logarithmic quasi-likelihood function. Following \cite{Yos11}, we introduce the quasi-likelihood ratio random field
\begin{equation}
\mbbz_{N}(u):=\exp\left\{ \mbbh_{N}(\theta_{0}+D_{N}^{-1}u) - \mbbh_{N}(\theta_{0}) \right\}
\nn%\label{hm:ZN_def}
\end{equation}
defined on the set $\mbbu_{N}=\mbbu_{N}(\tz):=D_{N}(\Theta-\tz)\subset\mbbr^{p}$. For convenience, we extend the domain of $\mbbz_{N}$ into the whole $\mbbr^{p}$ with $\mbbz_{N}\equiv 0$ on $\mbbr^{p}\setminus\mbbu_{N}$. 
Then, the quasi-posterior distribution $\bar{\Pi}_{N}(du|X^N)$ of the rescaled parameter $u=D_{N}(\theta-\tz)$ admits the density
\begin{equation}
u \mapsto \frac{\mbbz_{N}(u)\pi(\tz+D_{N}^{-1}u)}{\int \mbbz_{N}(v)\pi(\tz+D_{N}^{-1}v)\dif v}.
\nonumber
\end{equation}
Let $\Delta_{N}(\tz):=D_{N}^{-1}\p_{\theta}\mbbh_{N}(\tz)$ and denote by $\cip$ the convergence in probability.
For an $\mcf$-measurable random variable $(\mu,\Sigma)\in\mbbr^{p}\times\mbbr^{p\times p}$, we keep denoting by $N_{p}(\mu,\Sigma)$ the $\mcf$-conditional Gaussian distribution associated with the characteristic function $\xi\mapsto\E[\exp\{i\mu[\xi]-(1/2)\Sigma [\xi,\xi] \}]$.
Let $\|\nu\|_{\mathrm{TV}}=\sup_{|f|\le 1}|\int f(x)\nu(\dif x)|$ denote the total variation norm of a signed measure $\nu$. 

\begin{thm}
Under Assumptions \ref{hm:A1} to \ref{hm:A3}, we have
\begin{align}\label{hm:BvM}
\left\|\bar{\Pi}_{N}(\cdot|X^{N})-N_{p}\left(I(\tz)^{-1}\Delta_{N}(\tz),\, I(\tz)^{-1}\right)\right\|_{\mathrm{TV}} \cip 0.
\end{align}
\label{hm:thm_bvM}
\end{thm}

The proof of Theorem \ref{hm:thm_bvM} is given in Appendix \ref{app:bvm}.

%\begin{proof}
%Write  $\mbbz^{0}_{N}(u) =\exp\{ \Delta_{N}(\tz)[u] -(1/2)I(\tz)[u,u] \}$. It suffices to prove
%\begin{align}
%& \del_{N} := \int \bigg|\mbbz_{N}(u)\pi(\tz+D_{N}^{-1}u) - \mbbz_{N}^{0}(u)\pi(\tz) \bigg|\dif u \cip 0,
%\label{hm:thm_bvM-p6} \\
%& \bigg( \int \mbbz_{N}(u)\pi(\tz+D_{N}^{-1}u)\dif u \bigg)^{-1} = O_{p}(1),
%\label{hm:thm_bvM-p++1}
%\end{align}
%since the left-hand side of \eqref{hm:BvM} can be bounded from above by the quantity
%\begin{align}
%& \int\bigg|\frac{\mbbz_{N}(u)\pi(\tz+D_{N}^{-1}u)}{\int \mbbz_{N}(u)\pi(\tz+D_{N}^{-1}u)\dif u} - \frac{\mbbz_{N}^{0}(u)\pi(\tz)}{\int \mbbz^{0}_{N}(u)\pi(\tz)\dif u} \bigg| \dif u \nn\\
%&\lesssim O_{p}(1) \cdot \bigg( \int \mbbz_{N}(u)\pi(\tz+D_{N}^{-1}u) \dif u \bigg)^{-1}\del_{N}.
%\nonumber
%\end{align}
%The proof of \eqref{hm:thm_bvM-p6} and \eqref{hm:thm_bvM-p++1} are in Appendix \ref{app:bvm}.
%\end{proof}

%%%%%
%%%%%
\section{Large sample properties of the Markov chain Monte Carlo method}\label{sec:theory}

In this section, we will first introduce a general framework for developing some stability properties of MCMC algorithms, and then apply it to the  MCMC algorithms proposed in Section \ref{sec:mcmc}.
In particular, we show how to scale the proposals when the frequency of the data grows.

\subsection{Local consistency property}

For a moment we step away from the main context. Notations and terminologies in this section generally follow those of \cite{IH}.

\subsubsection{Some general notions}

Let $E$ be a normed space equipped with the Borel $\sigma$-algebra $\mathcal{E}$. 
For a probability measures $\mu$ and $\nu$ on $(E,\mathcal{E})$, let 
$$\|\mu-\nu\|_{\mathrm{BL}}=\sup_{f\in \mathrm{BL}_1}\left|\int f(x)\mu(\dif x)-\int f(x)\nu(\dif x)\right|$$
 be the bounded Lipschitz distance between $\mu$ and $\nu$ where 
$\mathrm{BL}_1$ denotes the set of any $\mathcal{E}$-measurable real-valued function $f$ such that 
$\sup_{x\in E}|f(x)|\le 1$ and $|f(x)-f(y)|\le \|x-y\|\ (x, y\in E)$. 

Let 
$\mathscr{E}_N=\{\mathscr{X}^{(N)},\mathfrak{A}^{(N)},\mathbf{P}_\theta^{(N)},\theta\in\Theta\}$  be a family of statistical experiments and $x^N$ be the corresponding observation. Here, the parameter space 
$\Theta\subset\mathbb{R}^p$ be an open set equipped with the Borel $\sigma$-algebra $\Xi$. 
Let $\mathbb{L}_N(\theta|x^N)$ be the $\mathscr{X}^{(N)}\otimes\Xi$-measurable quasi-likelihood function which approximates the true likelihood $\dif \mathbf{P}_\theta^{(N)}(x^N)$. 
Let $\Pi(\dif \theta)$ be the prior distribution for the experiments, and $\Pi_N(\dif\theta|x^N)$ be the quasi-posterior distribution defined by 
\begin{align*}
	\Pi_N(\dif\theta|x^N)=\frac{\mathbb{L}_N(\theta|x^N)\Pi(\dif\theta)}{\int_\Theta\mathbb{L}_N(\theta|x^N)\Pi(\dif\theta)}. 
\end{align*}

%In this paper, Markov chain Monte Carlo methods are considered. 
We shall refer to $\mathscr{M}_N$ the Markov chain Monte Carlo method for the experiment $\mathscr{E}_N$
if the algorithm generates a Markov chain $\{\theta^N_m;m=0,1,\dots\}$ for each $x^N$.
We shall refer to $F^N_M$ the empirical distribution of  $\{\theta^N_m;0\le m\le M-1\}$, that is, 
\begin{align*}
	F^N_M(A)=\frac{1}{M}\sum_{m=0}^{M-1}1_A(\theta_m^N)
\end{align*}
for each Borel set $A$ of $\mathbb{R}^p$. Let $D_N$ be a  non-degenerate $p\times p$-matrix for each $N\in\mathbb{N}$. 
In practice, we expect that Markov chain Monte Carlo methods are robust and does not collapse as $N\rightarrow\infty$. 
We formalize this favourable property as follows. 

\begin{definition}
	A family $\mathscr{M}_N$ is called locally consistent if  as $N, M\rightarrow\infty$ 
	\begin{align*}
		\|\Pi_N(\theta_0 + D_N^{-1} \dif u|x^N)-F^N_M(\theta_0 + D_N^{-1} \dif u)\|_{\mathrm{BL}}\rightarrow 0
	\end{align*}
	in $\mathbf{P}_{\theta_0}^{(N)}$-probability. 
\end{definition} 

In the next subsection, we introduce a sufficient condition for this property under the general framework. 

%\subsection{Sufficient condition for the local consistency property}
\subsubsection{Sufficient conditions}

Let $Q_N(\theta,\dif\vartheta|x^N)\ (n\in\mathbb{N})$ be a probability transition kernel 
from $\mathscr{X}^{(N)}\times\Theta$ to $\Theta$, 
and let $\Pi_N(\dif\theta|x^N)$ be a probability transition kernel from $\mathscr{X}^{(N)}$ to $\Theta$. 
Let $A_N:\mathscr{X}^{(N)}\times\Theta^2\rightarrow[0,1]$ be a measurable map. Let 
\begin{align*}
	P_N(\theta,\dif \vartheta|x^N)&=Q_N(\theta,\dif \vartheta|x^N)A_N(\theta,\vartheta|x^N)+R_N(\theta|x^N)\delta_\theta(\dif \vartheta),
\end{align*}
where 
\begin{align*}
	R_N(\theta|x^N)=1-\int_\Theta Q_N(\theta,\dif \vartheta|x^N)A_N(\theta,\vartheta|x^N). 
\end{align*}
As probability measures on $\Theta^2$, we assume that for each $x^N\in \mathscr{X}^{(N)}$, 
\begin{align*}
	\Pi_N(\dif\theta|x^N)Q_N(\theta,\dif \vartheta|x^N)A_N(\theta,\vartheta|x^N)
	=
	\Pi_N(\dif\vartheta|x^N)Q_N(\vartheta,\dif \theta|x^N)A_N(\vartheta,\theta|x^N). 
\end{align*}
Then for each $x^N$, the probability transition kernel $P_N(\theta,\dif\vartheta|x^N)$ is $\Pi_N(\dif\theta|x^N)$-reversible, that is, 
\begin{align*}
	\Pi_N(\dif\theta|x^N)P_N(\theta,\dif\vartheta|x^N)=\Pi_N(\dif\vartheta|x^N)P_N(\vartheta,\dif\theta|x^N)
\end{align*}
as probability measures on $\Theta^2$. 
Let $\mathscr{M}_N$ be the Markov chain Monte Carlo method associated with the transition kernel $P_N$, 
that is, for each $x^N$, the algorithm generates a Markov chain  $\{\theta_m^N\}_m$  with respect to the probability transition kernel $P_N(\theta,\dif\vartheta|x^N)$. 
For simplicity, the initial point is generated from the quasi-posterior distribution, that is, 
$\theta_0^N\sim \Pi_N(\dif\theta|x^N)$.

Let $q\in\mathbb{N}$. Let $\Pi(\dif u|s)=\pi(u|s)\dif u$ be a probability transition kernel from an open set $\mathcal{S}\subset \mathbb{R}^q$ to $\mathbb{R}^p$
and let $Q(u, \dif v|s)=q(u,v|s)\dif v$ be a probability transition kernel from $\mathcal{S}\times\mathbb{R}^p$ to $\mathbb{R}^p$. 
Let $A(u,v|s)$ be a $[0,1]$-valued function.
For each $s\in\mathcal{S}$ we define a probability transition kernel
\begin{align}
P(u,\dif v|s)&=Q(u,\dif v|s)A(u,v|s)+R(u|s)\delta_u(\dif v), \nn\\
R(u|s)&=1-\int Q(u,\dif v|s)A(u,v|s). 
\nonumber
\end{align}

\begin{assumption}[Regularity of the limit kernel]
For each $u,v\in\mathbb{R}^p$, $\pi(u|s), q(u,v|s)$ and $A(u,v|s)$ are continuous in $s$, and 
\begin{align*}
	\sup_{s\in K}\sup_{u,v\in\mathbb{R}^p}\pi(u|s)+q(u,v|s)<\infty
\end{align*}
for any compact set $K$ of $\mathcal{S}$. 
\label{as:kernel}
\end{assumption}

Let 
$s_N:\mathscr{X}^{(N)}\rightarrow \mathcal{S}$ be a sequence of random variables. The following proposition illustrates a sufficient condition for 
local consistency. See Appendix \ref{app:lc} for the proof. 
In the following, we call that a transition probability kernel $P$ is ergodic if it is irreducible and positive recurrent. 

\begin{proposition}\label{prop:mtk}
Suppose that $P(u,\dif v|s)$ is ergodic with invariant probability measure $\Pi(\dif u|s)$ for each $s\in\mathcal{S}$. 
Suppose that $s_N$ is $\mathbf{P}_{\theta_0}^{(N)}$-tight for each $\theta_0\in\Theta$. 
	Then, under Assumption \ref{as:kernel}, the family $\mathscr{M}_N$ is locally consistent if $A_N(\theta_0 + D_N^{-1} u,\theta_0 + D_N^{-1} v|x^N)- A(u,v|s_N)\rightarrow 0$ in probability for each $u, v\in\mathbb{R}^p$ and 
\begin{align}\label{BvMMH}
	\|\Pi_N(\theta_0 + D_N^{-1}\dif u|x^N)Q_N(\theta_0 + D_N^{-1} u,\theta_0 + D_N^{-1}\dif v|x^N)-\Pi(\dif u|s_N)Q(u,\dif v|s_N)\|_{\mathrm{TV}}
\end{align}
converges in $\mathbf{P}_{\theta_0}^{(N)}$-probability to $0$. 
\end{proposition}

\subsection{Main result}

We now go back to the main context given in Section \ref{sec:setup}.
In Algorithm \ref{alg:1}, the proposal distribution for observation $X^N$ is 
\begin{equation}
Q_N(\theta_0 + D_N^{-1} u,\theta_0 + D_N^{-1} \dif v|X^N)=Q(u,\dif v|s_N)=N_p(0,\Sigma)
\nonumber
\end{equation}
for some non-degenerate matrix $\Sigma\in\mathbb{R}^{p\times p}$. 
We will set $s_N=(\Delta_{N}(\theta_0), I(\theta_0), I^*(\theta_0))$ where $I^*(\tz)$ is defined in Section \ref{hm:sec_conv-ar}.
By this setting of the proposal distribution, the following convergence is equivalent to  the condition (\ref{BvMMH}):  
\begin{align*}
	\|\Pi_N(\theta_0 + D_N^{-1}\dif u|X^N)-\Pi(\dif u|s_N)\|_{\mathrm{TV}}=o_p(1). 
\end{align*}
The acceptance probability of  Algorithm \ref{alg:1} is 
\begin{align*}
		A_N(\theta,\vartheta|X^N)=\mathbb{E}\left[\left.\min\left\{1,\frac{\mathbb{L}(\vartheta|X^N,V^N)\pi(\vartheta)}{\mathbb{L}_N(\theta|X^N,V^N)\pi(\theta)}\right\}\right|X^N,\theta\right]. 
\end{align*}
Let $s=(\Delta, I, I^*)\in \mathbb{R}^p\times \mathbb{R}^{p\times p}\times\mathbb{R}^{p\times p}$. 
We assume that $I$ and $I^*$ are $p\times p$-symmetric matrices. 
In the following theorem, we will show that the scaled version of the acceptance ratio converges to
\begin{align*}
		A(u,u^*|\Delta,I,I^*)=\mathbb{E}[\min\left\{1,\eta\right\}]
\end{align*}
where 
\begin{align*}
	\eta=\Delta[u^*-u]+W[u^*-u]-I[(u^*)^{\otimes 2}-u^{\otimes 2}]-\frac{1}{2}I^*[(u^*-u)^{\otimes 2}]
\end{align*}
with $W\sim N(0, I^*)$. 
By showing this, we prove that Algorithm \ref{alg:1} generates a locally consistent family of Markov chain Monte Carlo methods. 
In particular, the algorithm will not collapse in the large data limit.

\begin{theorem}[Metropolis-within-Gibbs algorithm for the jump process]\label{prop:mwg}
Under Assumptions \ref{hm:A1} to \ref{hm:A3}, 
	\begin{align*}
		A_N(\theta_0+D_N^{-1}u,\theta_0+D_N^{-1} v|X^N)-A(u,v|\Delta_{N}(\theta_0), I(\theta_0), I^*(\theta_0))\rightarrow 0
	\end{align*}
	in probability for each $u, v\in\mathbb{R}^p$. 
	In particular, the family of Markov chain Monte Carlo methods $\mathscr{M}_N$ defined by Algorithm \ref{alg:1} is locally consistent. 
\end{theorem}

The proof of the theorem is given in Section \ref{sec:proof_mwg}.

\section{Simulations}\label{sec:simos}

\subsection{Simulated Example}

We generated seven data sets from the time discretized models with
$a(x,\alpha)=\alpha_1(x-\alpha_2)$, $c(x,\gamma)=\exp(\gamma \cos(x))$,
$\beta=1.5$, $h_N=1/N$ and $N\in\{10,50,100,250,500,1000,2000\}$.
$\alpha_1, \alpha_2, \gamma$ are independently standard normal variables in the prior.
Algorithm \ref{alg:1} was run
for $10,000$ iterations. All simulations are repeated $100$ times.
The results are given in Figure \ref{fig:accrate}.

Figure \ref{fig:accrate} illustrates the average acceptance rate against the number of data.
As can be seen the average acceptance rate is stable from small data to large data.
In particular, the data increases the algorithm does not collapse and the acceptance rate is very reasonable.
The run time of the algorithm for $N=2000$ is only about five minutes and was coded in R (code is available upon request); the
code could be substantially improved to further reduce the computation time. Note that the update on $V^N|\cdots$
can be parallelized to improve the running speed.

\begin{figure}[!h]\centering
\includegraphics[width=10cm]{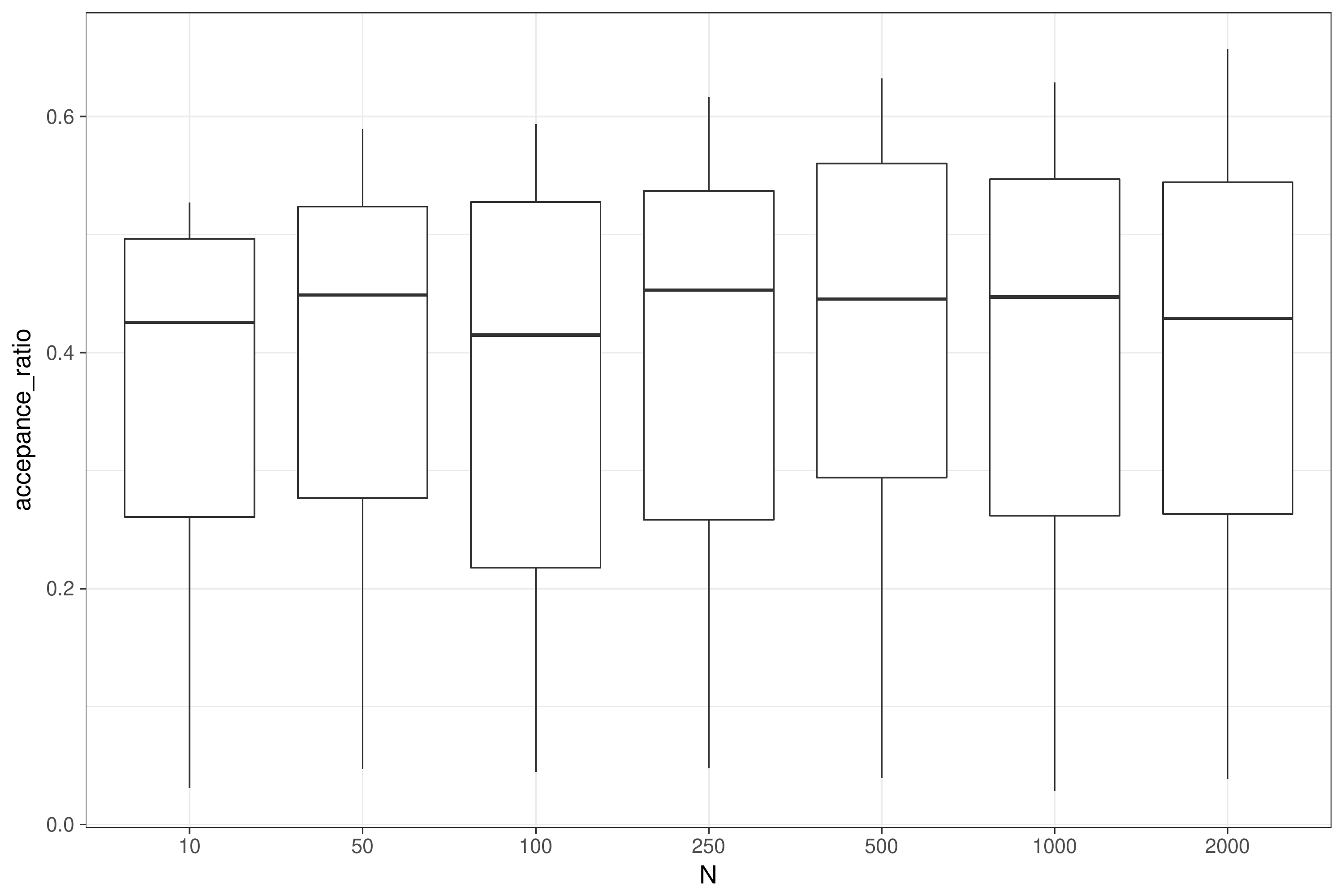}
\caption{Acceptance Rate against $N$. The
MCMC algorithm was run for $10,000$ iterations and all runs repeated $100$ times.
%The dotted lines are average $\pm$ twice the standard deviation of the acceptance rate across the 20 runs.
}
\label{fig:accrate}
\end{figure}

\subsection{Application to IBM Stock Data}

We return to the data of NYSE. %Figure \ref{fig:data}.
In the model, we set the same model with above, and 
$\beta=1.411$, $h_N=T/N$, where $T$ is measured in minutes ($T=390\times 3$) and
$N=1156$. The stable index $\beta$ was estimated by applying `stableFit' function in R-package `fBasics' \cite{fbasics} to $\{\Delta_n^N X\}_{n=1,\ldots, N}$. The data contains some missing values which are ignored for the purposes of the analysis.
$\alpha_1, \alpha_2,\gamma$ are independently normal variables in the prior with mean $0$ and standard deviation $1$.
Algorithm \ref{alg:1} was run
for 100,000 iterations. %and we multiply $D_N$ by 0.5. 
The algorithm ran in approximately 1 hour
and the acceptance rate for the move on the parameters was 0.34.

In Figures \ref{fig:mcmc} and \ref{fig:diff} we can observe our results. 
Figure \ref{fig:mcmc} is a log-likelihood and 
time average drift and  jump coefficients: 
$$
\frac{1}{N}\sum_{n=1}^Na(X_{(n-1),h},\alpha),\ 
\frac{1}{N}\sum_{n=1}^Nc(X_{(n-1),h},\gamma).
$$
%In Figure \ref{fig:mcmc} the performance of the MCMC algorithm can be observed. 
For
this reasonable size data set, the algorithm performs well, with good mixing over the parameters.
The acceptance rate is very reasonable as is the run-time - recall one can improve the code or coding language.

Figure \ref{fig:diff} is a p-p plot of the posterior expected value of the standardized residual:
$$
\frac{\Delta_n^N X-ha(X_{(n-1)h},\alpha)h}{c(X_{(n-1)h},\gamma)h^{1/\beta}}
$$
against $\beta$-stable distribution. This provides an idea of the ability of this model to fit the real NYSE data. We can see that the model, to some extent, can exhibit the behaviour
in the real data.

\begin{figure}[h]
\centering 
\subfigure[Log-Likelihood]{{\includegraphics[width=0.3\textwidth,height=5.5cm]{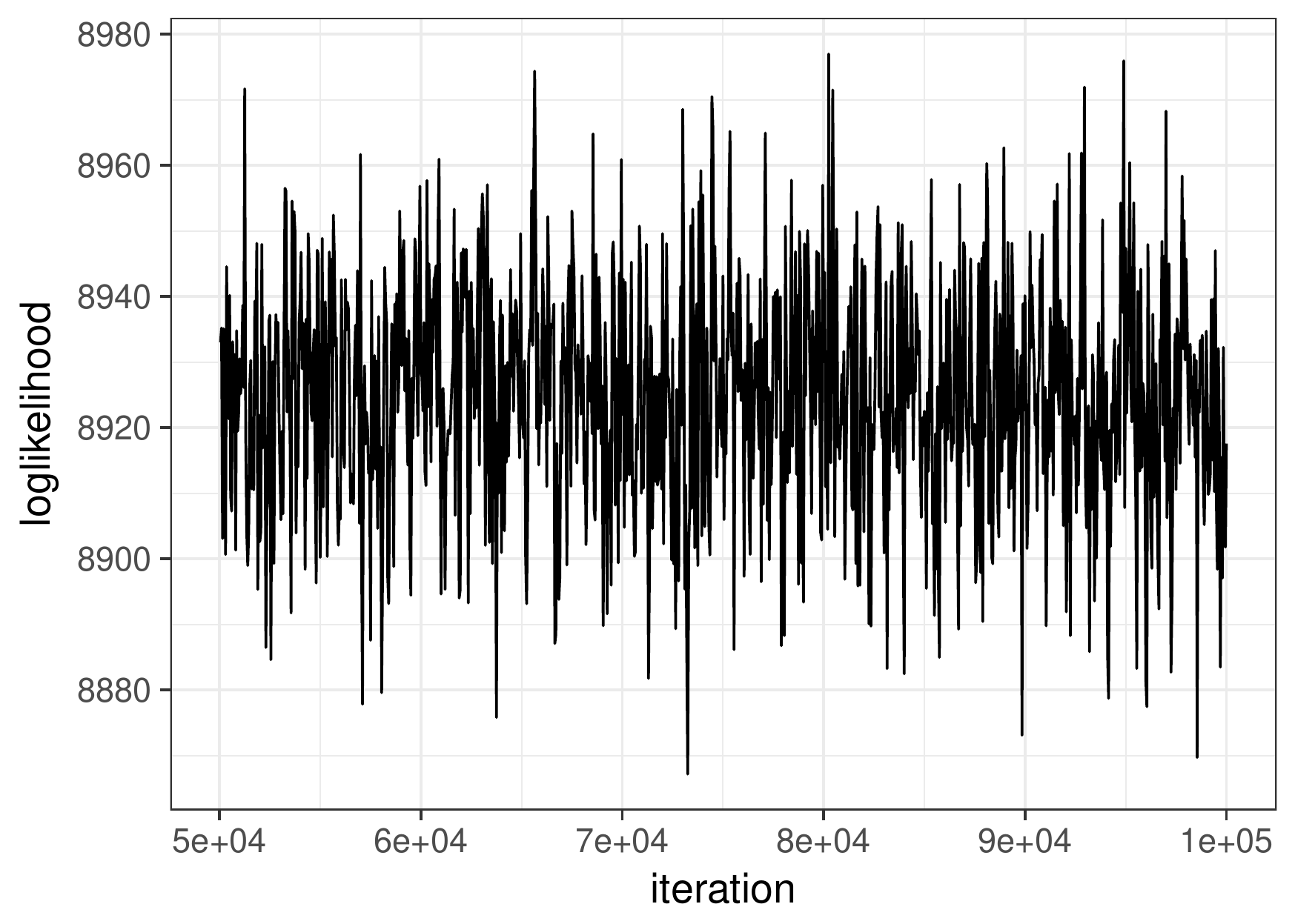}}}
\subfigure[Average drift coefficient]{{\includegraphics[width=0.3\textwidth,height=5.5cm]{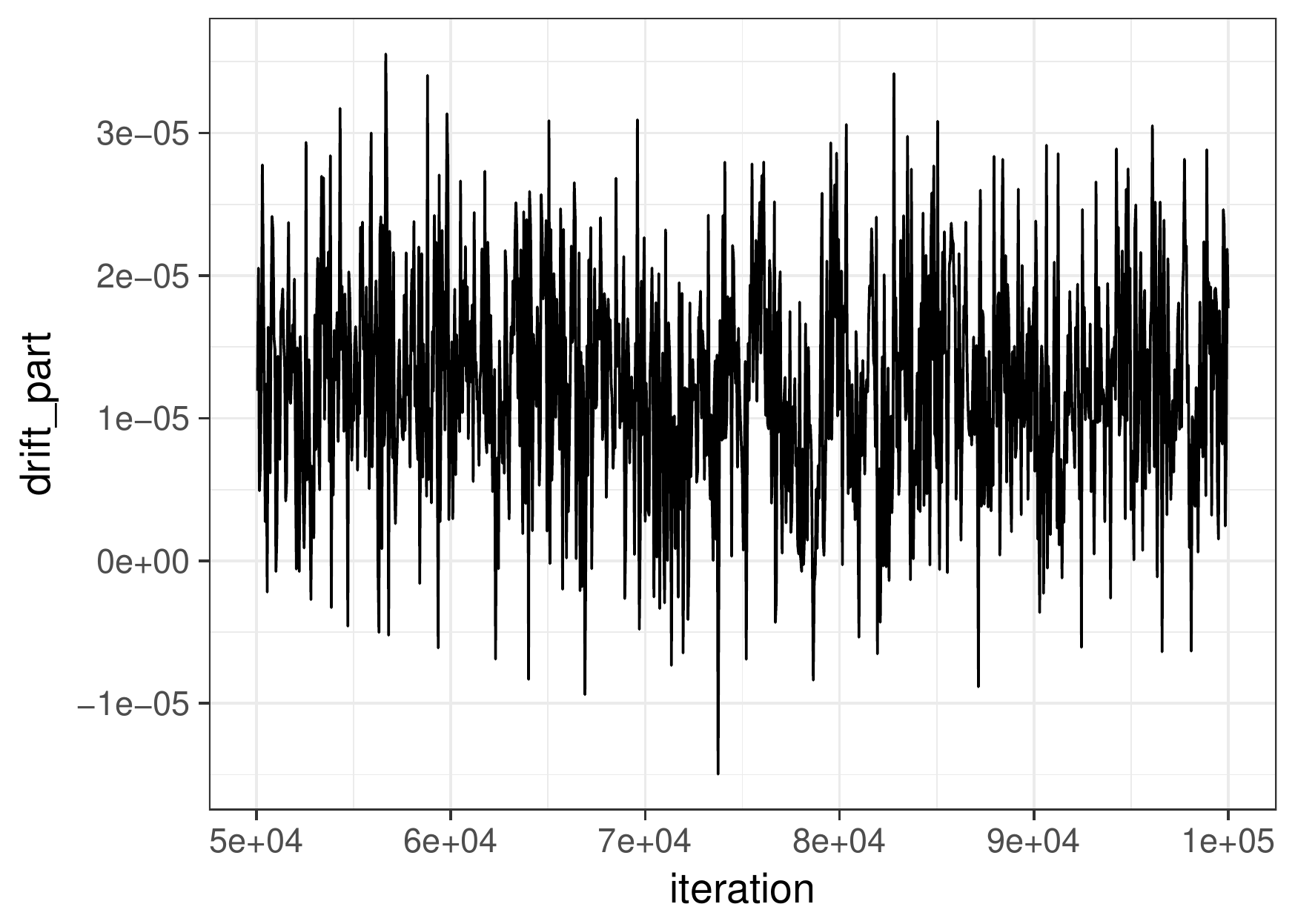}}}
\subfigure[Average jump coefficient]{{\includegraphics[width=0.3\textwidth,height=5.5cm]{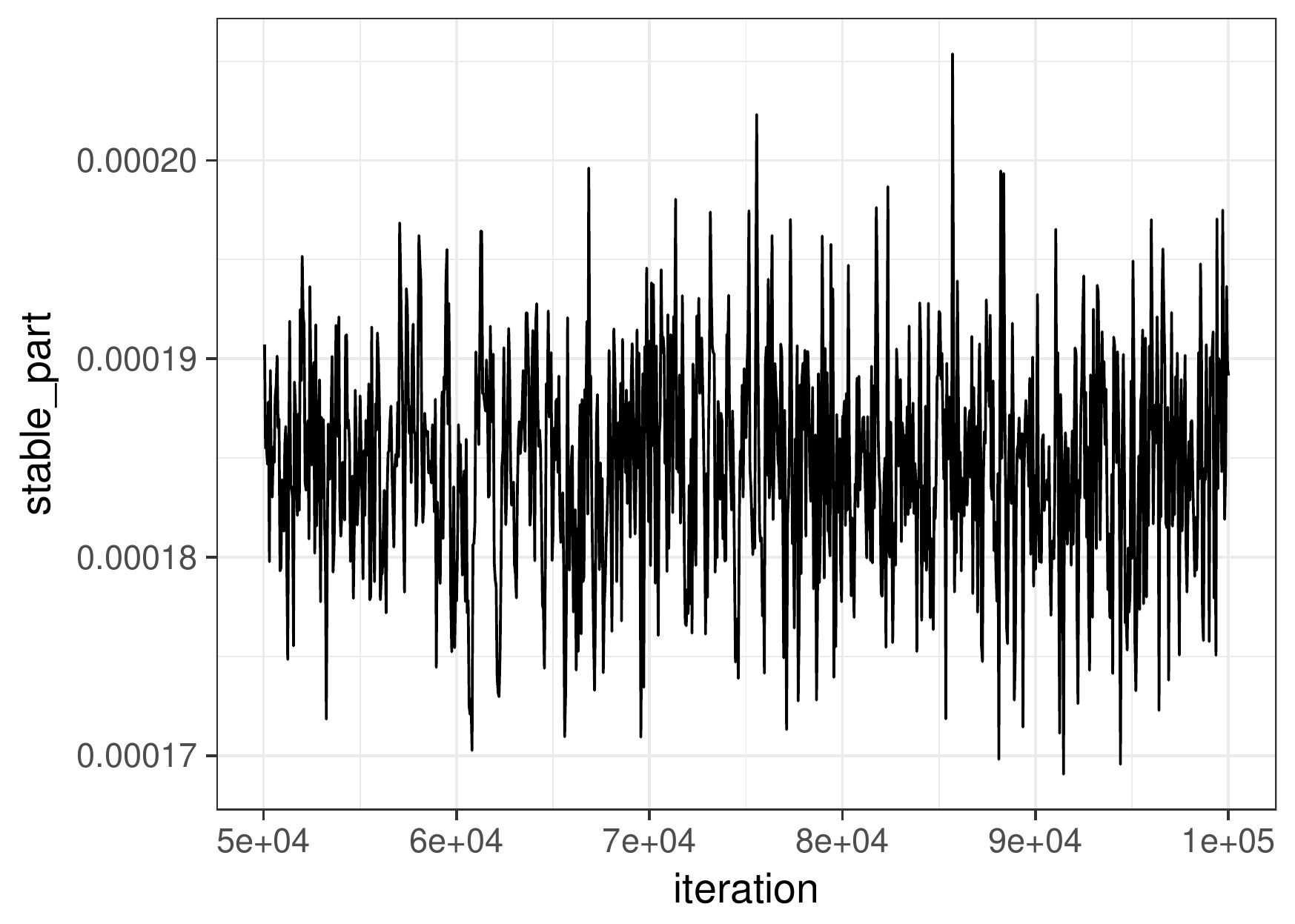}}}
\caption{MCMC Plots. The MCMC algorithm was run for 100,000 iterations and the last half with every $50$ iterations is displayed.}
\label{fig:mcmc} 
\end{figure}

\begin{figure}[!h]\centering
\includegraphics[width=10cm,height=6cm,angle=0]{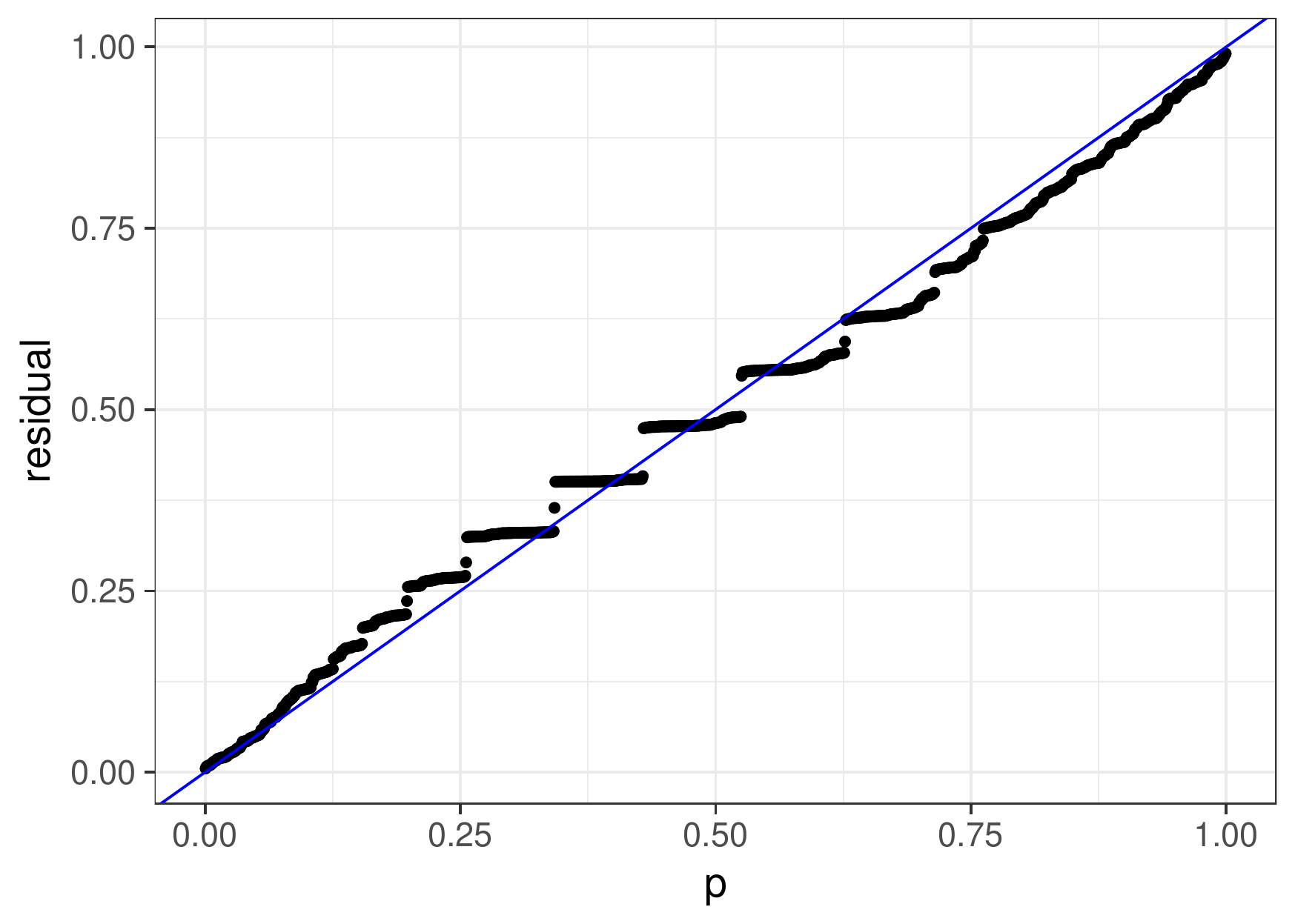}
\caption{Posterior expected p-p plot for the standardized residual against $\beta$-stable distribution}
\label{fig:diff}
\end{figure}

\subsubsection*{Acknowledgements}
We acknowledge JST CREST Grant Number JPMJCR14D7, Japan, for supporting the research.
AJ was additionally supported by Ministry of Education AcRF tier 2 grant,
R-155-000-161-112, and KK was additionally supported by JSPS KAKENHI Grant Number JP16K00046.

\appendix

\section{Proofs for the Bernstein-von Mises theorem}
\label{app:bvm}

The purpose of this section is to prove Theorem \ref{hm:thm_bvM}. 
First, we introduce the quadratic random field
\begin{equation}
\mbbz^{0}_{N}(u) =\exp\bigg( \Delta_{N}(\tz)[u] - \frac{1}{2}I(\tz)[u,u] \bigg).
\nonumber
\end{equation}
By integrating the Gaussian density,
\begin{equation}
\bigg( \int \mbbz_{N}^{0}(u)\pi(\tz)\dif u \bigg)^{-1} = \frac{1}{\pi(\tz)}\exp\bigg( \frac{1}{2}I(\tz)^{-1}[\Delta_{N}(\tz)^{\otimes 2}] \bigg) = O_{p}(1).
\nonumber
\end{equation}
To complete the proof it suffices to show
\begin{align}
& \del_{N} := \int \bigg|\mbbz_{N}(u)\pi(\tz+D_{N}^{-1}u) - \mbbz_{N}^{0}(u)\pi(\tz) \bigg|\dif u \cip 0,
\label{hm:thm_bvM-p6}
\end{align}
since the left-hand side of \eqref{hm:BvM} can be bounded from above by the quantity
\begin{align}
\int\bigg|\frac{\mbbz_{N}(u)\pi(\tz+D_{N}^{-1}u)}{\int \mbbz_{N}(u)\pi(\tz+D_{N}^{-1}u)\dif u} - \frac{\mbbz_{N}^{0}(u)\pi(\tz)}{\int \mbbz^{0}_{N}(u)\pi(\tz)\dif u} \bigg| \dif u
\le 2 \bigg( \int \mbbz_{N}^{0}(u)\pi(\tz) \dif u \bigg)^{-1}\del_{N}.
\nonumber
\end{align}
We only prove \eqref{hm:thm_bvM-p6} for $\beta\in(1,2)$, since the remaining case of $\beta=1$, where we have the single localization rate $\sqrt{N}(=\sqrt{N}h^{1-1/\beta})$, is completely analogous and simpler.

\medskip

First we will introduce a good event $G_{N}\in\mcf$ whose probability can get arbitrarily close to $1$ for $N\to\infty$ (see \eqref{hm:def_G_N} below).
Let
\begin{align}
\mbbz_{1,N}(u_{1};\al) &:= \exp\bigg\{\mbbh_{N}\bigg(\al, \gam_{0}+\frac{u_{1}}{\sqrt{N}}\bigg) - \mbbh_{N}(\al, \gam_{0}) \bigg\}, \label{def_Z1n}\\
\mbbz_{2,N}(u_{2}) &:= \exp\bigg\{ \mbbh_{N}\bigg(\al_{0} + \frac{u_{2}}{\sqrt{N}h^{1-1/\beta}}, \gam_{0}\bigg) - \mbbh_{N}(\al_{0}, \gam_{0}) \bigg\},
\label{def_Z2n}
\end{align}
both being defined to be zero for $u=(u_{1},u_{2})\in\mbbr^{p}\setminus\mbbu_{N}$.
Also define
\begin{align}
\mbby_{1,N}(\theta) &:=\frac{1}{N}\{\mbbh_{N}(\al,\gam) - \mbbh_{N}(\al,\gam_{0})\}, \nn\\
\mbby_{1}(\gam) &:= \frac{1}{T}\int_{0}^{T}\int\bigg[ 
\log\bigg\{\frac{c(X_{t},\gam_{0})}{c(X_{t},\gam)}\phi_{\beta}\bigg(\frac{c(X_{t},\gam_{0})}{c(X_{t},\gam)}z\bigg)\bigg\}
-\log\phi_{\beta}(z)\bigg]\phi_{\beta}(z)\dif z\dif t,
\nn \\
\mbby_{2,N}(\al) &:= \frac{1}{Nh^{2(1-1/\beta)}} \{\mbbh_{N}(\al,\gam_{0}) - \mbbh_{N}(\al_{0},\gam_{0})\}, \nn\\
\mbby_{2}(\al) &:= \frac{-1}{2T}\int_{0}^{T}\bigg( \frac{a(X_{t},\al_{0})-a(X_{t},\al)}{c(X_{t},\gam_{0})} \bigg)^{2}
\dif t \cdot \int \bigg(\frac{\p\phi_{\beta}}{\phi_{\beta}}(z)\bigg)^{2}\phi_{\beta}(z)\dif z.
\nn
\end{align}
The random functions $\mbby_{1}$ and $\mbby_{2}$ represent quasi-Kullback-Leibler divergences for estimating $\gam$ and $\al$, respectively;
we can estimate $\gam_{0}$ more quickly than $\al_{0}$ in case of $\beta\in(1,2)$.
For any matrix $A$ we will write $|A|$ for the Frobenius norm of $A$.
For later use we mention the following statements, which are given in \cite[Section 6]{masuda16.sqmle} or can be directly deduced from the arguments therein.

\begin{itemize}

\item We have (recall \eqref{hm:amn}) under $\tilde{\pr}$
\begin{align}
& \left(\Delta_{N}(\tz),\, -D_{N}^{-1}\p_{\theta}^{2}\mbbh_{N}(\tz)D_{N}^{-1}\right) \overset{\mcl}\to \left( I(\tz)^{1/2}\eta,\, I(\tz)\right),
\nn\\%\label{hm:thm_bvM-p3}\\
& 
\bar{r}_{N}:=\sup_{\theta}\left| D_{N}^{-1}\p_{\theta}^{3}\mbbh_{N}(\theta)D_{N}^{-1} \right| = O_{p}(1).
%\bar{r}_{N}:=\sup_{\theta}\max_{1\le k\le p}\left| \p_{\theta_{k}} \left(D_{N}^{-1}\p_{\theta}^{2}\mbbh_{N}(\theta)D_{N}^{-1}\right) \right| = O_{p}(1).
\nn%\label{hm:thm_bvM-p4}
\end{align}

\item There exists an a.s. positive random variable $\chi_{0}$ such that for each $\kappa>0$,
\begin{equation}
\sup_{\gam;\, |\gam-\gam_{0}|\ge\kappa}\mbby_{1}(\gam) \vee 
\sup_{\al;\, |\al-\al_{0}|\ge\kappa}\mbby_{2}(\al) \le -\chi_{0}\kappa^{2}\qquad \text{a.s.}
\nonumber
\end{equation}

\item In case of $\beta>1$, we have
\begin{align}
& \sup_{\al}\bigg| \frac{1}{\sqrt{N}}\p_{\gam}\mbbh_{N}(\al,\gam_{0}) \bigg| = O_{p}(1), \nn \\
& \sup_{\al} \bigg| -\frac{1}{N}\p_{\gam}^{2}\mbbh_{N}(\al,\gam_{0}) - C_{\gam}(\beta)\Sigma_{T,\gam}(\gam_{0}) \bigg|=o_{p}(1).
\nonumber
\end{align}

\end{itemize}
In addition, we will also need the following uniform laws of large numbers with convergence rates:
\begin{itemize}
\item There exists a constant $q\in(0,1)$ for which
\begin{equation}
(\sqrt{N})^{q}\sup_{\theta}|\mbby_{1,N}(\theta)-\mbby_{1}(\gam)| \vee (\sqrt{N}h^{1-1/\beta})^{q}\sup_{\al}|\mbby_{2,N}(\al)-\mbby_{2}(\al)| \cip 0.
\label{hm:thm_bvM-p++3}
\end{equation}
\end{itemize}
A sketch of derivation of \eqref{hm:thm_bvM-p++3} will be given at the end of this section.

\medskip

Set
\begin{equation}
\ep_{N}=N^{-c} \qquad\text{for}\qquad 0<c<\frac{q}{2}\bigg(\frac{1}{\beta}-\frac{1}{2}\bigg)
\label{hm:thm_bvM-p++2}
\end{equation}
with $q\in(0,1)$ given in \eqref{hm:thm_bvM-p++3}. Then, since
\begin{equation}
\sqrt{N}h^{1-1/\beta} \sim N^{1/\beta-1/2}
\label{hm:thm_bvM-p+7}
\end{equation}
up to a multiplicative constant, we see that $\ep_{N}\sqrt{N}h^{1-1/\beta} \uparrow \infty$ and $\ep_{N}^{-2}(\sqrt{N}h^{1-1/\beta})^{-q}=O(1)$;
since $\beta\ge 1$, the latter implies that $\ep_{N}^{-2}N^{-q/2}=O(1)$. It is straightforward to deduce from \eqref{hm:thm_bvM-p++3} that
\begin{equation}
\ep_{N}^{-2}\sup_{\theta}|\mbby_{1,N}(\theta)-\mbby_{1}(\gam)| \vee \ep_{N}^{-2}\sup_{\al}|\mbby_{2,N}(\al)-\mbby_{2}(\al)| \cip 0.
\nonumber
\end{equation}
Let $\lam_{\min}(A)$ denote the minimum eigenvalue of a square matrix $A$.
We now introduce the event $G_{N}=G_{N}(M,\lam)$ for positive constants $M$ and $\lam$:
\begin{align}
G_{N} &:= \bigg\{
|\Delta_{N}(\tz)| \vee \sup_{\al}\bigg| \frac{1}{\sqrt{N}}\p_{\gam}\mbbh_{N}(\al,\gam_{0}) \bigg| \le M,
\nn\\
&{}\qquad \left| -D_{N}^{-1}\p_{\theta}^{2}\mbbh_{N}(\tz)D_{N}^{-1} - I(\tz) \right| \vee \sup_{\al} \bigg| -\frac{1}{N}\p_{\gam}^{2}\mbbh_{N}(\al,\gam_{0}) - C_{\gam}(\beta)\Sigma_{T,\gam}(\gam_{0}) \bigg| <\lam,\nn\\
&{}\qquad 
\ep_{N}^{-2}\sup_{\theta}|\mbby_{1,N}(\theta)-\mbby_{1}(\gam)| \vee \ep_{N}^{-2}\sup_{\al}|\mbby_{2,N}(\al)-\mbby_{2}(\al)| < \frac{\lam}{2},
\nn\\
&{}\qquad \lam_{\min}(I_{\al}(\tz)) \wedge \lam_{\min}(I_{\gam}(\tz)) \ge 4\lam, \quad \chi_{0}\ge 4\lam, \quad \overline{r}_{N}<\frac{3\lam}{\ep_{N}} \bigg\}.
\label{hm:def_G_N}
\end{align}
Given any $\ep'>0$, we can find a triple $(M_{1},\lam_{1},\ep_{1})$ and an $N_{1}\in\mbbn$ such that
\begin{equation}
\sup_{N\ge N_{1}}\pr\left\{G_{N}(M,\lam)^{c}\right\}<\ep'
\nonumber
\end{equation}
holds for every $M\ge M_{1}$ and $\lam\in(0,\lam_{1}]$.
Since the objective here is the convergence in probability, we may and do focus on the event $G_{N}$ with $M$ and $\lam$ being sufficiently large and small, respectively.

\medskip

We divide the domain of the integration in the definition of $\del_{N}$: $\mbbr^{p}=A_{N}\sqcup A_{N}^{c}$ where
\begin{equation}
A_{N}:=\left\{u;\, |D_{N}^{-1}u|\le \ep_{N} \right\},
\nn%\label{hm:def_An.ep}
\end{equation}
and then denote the associated integrals by $\del'_{N}$ and $\del''_{N}$, respectively:
$\del_{N} = \del'_{N} + \del''_{N}$ with
\begin{align}
\del'_{N} &:= \int_{A_{N}} \bigg|\mbbz_{N}(u)\pi(\tz+D_{N}^{-1}u) - \mbbz_{N}^{0}(u)\pi(\tz) \bigg|\dif u, \nn\\
\del''_{N} &:= \int_{A_{N}^{c}} \bigg|\mbbz_{N}(u)\pi(\tz+D_{N}^{-1}u) - \mbbz_{N}^{0}(u)\pi(\tz) \bigg|\dif u.
\nonumber
\end{align}
We will deal with these terms separately.

\medskip

First we show that $\del'_{N}\cip 0$. We have $\del'_{N}\le\del'_{1,N}+\del'_{2,N}$, where
\begin{align}
\del'_{1,N} &:= \int_{A_{N}} \mbbz_{N}(u)\left| \pi(\tz+D_{N}^{-1}u)-\pi(\tz)\right|\dif u,
\nn\\
\del'_{2,N} &:= \pi(\tz)\int_{A_{N}}\left| \mbbz_{N}(u)-\mbbz_{N}^{0}(u) \right|\dif u.
\nonumber
\end{align}
By the third order Taylor expansion
\begin{align}
\mbbz_{N}(u) &= \exp\bigg\{
\Delta_{N}(\tz)[u] - \frac{1}{2}\bigg( I(\tz) + \left( -D_{N}^{-1}\p_{\theta}^{2}\mbbh_{N}(\tz)D_{N}^{-1} - I(\tz) \right)
\nn\\
&{}\qquad -\frac{1}{3}
\left(D_{N}^{-1}\p_{\theta}^{3}\mbbh_{N}(\check{\theta}_{n}(u))D_{N}^{-1}\right)[D_{N}^{-1}u]
\bigg)[u,u]\bigg\}
\nn%\label{hm:thm_bvM-p++1.0}
\end{align}
for some random point $\check{\theta}_{n}(u)$ on the segment joining $\tes$ and $\tz$,
we see that $\mbbz_{N}(u) \le \exp(M|u| - \lam|u|^{2})$ for $u\in A_{N}$ on $G_{N}$.
Hence, under Assumption \ref{hm:A3} we obtain
\begin{equation}
\del'_{1,N} \le \sup_{u \in A_{N}} \left| \pi(\tz+D_{N}^{-1}u)-\pi(\tz)\right|
\int e^{M|u| - \lam|u|^{2}} \dif u \cip 0.
\nonumber
\end{equation}
To handle $\del'_{2,N}$, we introduce the random function
\begin{equation}
M_{N}(u) := I_{G_{N}}(\omega) I_{A_{N}}(u) \left| \mbbz_{N}(u)-\mbbz_{N}^{0}(u) \right|.
\nonumber
\end{equation}
To deduce the required convergence
\begin{equation}
\overline{M}_{N}:=\int M_{N}(u)\dif u \cip 0,
\label{hm:thm_bvM-p++1}
\end{equation}
we make use of the subsequence argument: fix any infinite sequence $\mbbn'\subset\mbbn$. In view of the estimate
\begin{align}
M_{N}(u) &\le I_{G_{N}}(\omega) I_{A_{N}}(u)\left( \mbbz_{N}(u) + \mbbz_{N}^{0}(u)\right) \nn\\
&\le e^{M|u| - \lam|u|^{2}} + e^{M|u| - 2\lam|u|^{2}} \lesssim e^{M|u| - \lam|u|^{2}} \qquad \text{a.s.}
%&\le \exp(M|u| - \lam|u|^{2}) + \exp(M|u| - 2\lam|u|^{2}) \lesssim \exp(M|u| - \lam|u|^{2})
\nonumber
\end{align}
and the dominated convergence theorem, it suffices to show that there exists a further subset $\mbbn''=\{N''\}\subset\mbbn'$ along which $M_{N''}(u) \to 0$ a.s. for each $u$.
We have
%Writing $r_{N}(u) = \log\mbbz_{N}(u) - \log\mbbz^{0}_{N}(u)$, we can estimate as
\begin{equation}
%|r_{N}(u)| 
\left| \log\mbbz_{N}(u) - \log\mbbz^{0}_{N}(u) \right|
\lesssim |u|^{2}\bigg(
\left| -D_{N}^{-1}\p_{\theta}^{2}\mbbh_{N}(\tz)D_{N}^{-1} - I(\tz) \right| + \bar{r}_{N}\ep_{N}
\bigg) =: |u|^{2}R_{N}.
\nonumber
\end{equation}
Since $| -D_{N}^{-1}\p_{\theta}^{2}\mbbh_{N}(\tz)D_{N}^{-1} - I(\tz) | \vee \bar{r}_{N}\ep_{N} \to 0$, 
it is possible to pick a further subset $\mbbn''=\{N''\}\subset\mbbn'$ along which $R_{N''}\to 0$ a.s.;
note that the random sequence $(R_{N})$ is free from the variable $u$.
With this $\mbbn''$, for each $u$ we have
\begin{align}
M_{N''}(u) &\le \mbbz_{N''}^{0}(u)\left| \exp\left\{\log\mbbz_{N}(u) - \log\mbbz^{0}_{N}(u)\right\}-1 \right| \nn\\
&\le \mbbz_{N''}^{0}(u) |u|^{2}R_{N''} e^{|u|^{2}R_{N''}} \to 0\qquad \text{a.s.},
\nonumber
\end{align}
followed by $\overline{M}_{N''}\to 0$ a.s., hence by \eqref{hm:thm_bvM-p++1}. We conclude that $\del'_{N}\cip 0$.

\medskip

Now we turn to the proof of $\del''_{N}\cip 0$, still focusing on the event $G_{N}$ of \eqref{hm:def_G_N}.
Note that $A_{N}^{c}\downarrow\emptyset$ under \eqref{hm:thm_bvM-p++2} since $|D_{N}^{-1}u|>\ep_{N}$ implies that $|u| \gtrsim \ep_{N}\sqrt{N}h^{1-1/\beta}\uparrow\infty$.
Under the condition \eqref{hm:thm_bvM-p++2} we have $\ep_{N}/|D_{N}^{-1}| \sim C\ep_{N}\sqrt{N}h^{1-1/\beta} \to \infty$, so that
\begin{equation}
\int_{A_{N}^{c}}\mbbz_{N}^{0}(u)\dif u \le \int_{|u| \ge \ep_{N}/|D_{N}^{-1}|} \exp(M|u| - 2\lam|u|^{2}) \dif u \to 0
\nonumber
\end{equation}
on $G_{N}$.
Since $\pi$ is bounded, we are left to show that 
\begin{equation}
\int_{A_{N}^{c}}\mbbz_{N}(u)\dif u \cip 0.
\nn%\label{hm:thm_bvM-p+1}
\end{equation}
Write $u=(u_{2},u_{1})\in\mbbr^{p_{\al}}\times\mbbr^{p_{\gam}}$ and observe that (recall \eqref{def_Z1n} and \eqref{def_Z2n})
\begin{equation}
\mbbz_{N}(u) \le \bigg(\sup_{\al}\mbbz_{1,N}(u_{1};\al)\bigg) \mbbz_{2,N}(u_{2}) =: \mbbz_{1,N}(u_{1})\mbbz_{2,N}(u_{2}).
\nonumber
\end{equation}
Since $A_{N}^{c}\subset \{u;\,|u_{2}|\ge (\ep_{N}/2)\sqrt{N}h^{1-1/\beta}\} \cup \{u;\,|u_{1}|\ge (\ep_{N}/2)\sqrt{N}\}$, we obtain
\begin{align}
\int_{A_{N}^{c}}\mbbz_{N}(u)\dif u &\le 
\int_{|u_{1}|\ge(\ep_{N}/2)\sqrt{N}} \mbbz_{1,N}(u_{1})\dif u_{1} \cdot \int \mbbz_{2,N}(u_{2})\dif u_{2} \nn\\
&{}\qquad + \int \mbbz_{1,N}(u_{1})\dif u_{1} \cdot \int_{|u_{2}|\ge (\ep_{N}/2)\sqrt{N}h^{1-1/\beta}}\mbbz_{2,N}(u_{2})\dif u_{2}.
\nn%\label{hm:thm_bvM-p+2}
\end{align}
We will only show that
\begin{align}
\int_{|u_{1}|\ge(\ep_{N}/2)\sqrt{N}} \mbbz_{1,N}(u_{1})\dif u_{1} &\cip 0,
\label{hm:thm_bvM-p+3} \\
\int \mbbz_{1,N}(u_{1})\dif u_{1} &= O_{p}(1).
\label{hm:thm_bvM-p+6}
\end{align}
Indeed, in view of the definition \eqref{hm:def_G_N} of the good event $G_{N}$, in order to deduce
\begin{equation}
\int_{|u_{2}|\ge(\ep_{N}/2)\sqrt{N}h^{1-1/\beta}} \mbbz_{2,N}(u_{2})\dif u_{2} \cip 0 \quad \text{and} \quad
\int \mbbz_{2,N}(u_{2})\dif u_{2} = O_{p}(1),
\nonumber
\end{equation}
we can follow exactly the same route as in the proofs of \eqref{hm:thm_bvM-p+3} and \eqref{hm:thm_bvM-p+6} below,
with $\mbby_{1,N}(\theta)$ and $\mbby_{1}(\gam)$ replaced by $\mbby_{2,N}(\al)$ and $\mbby_{2}(\al)$, respectively.

On $G_{N}$ we have $\ep_{N}^{-2}\sup_{\gam;\, |\gam-\gam_{0}|\ge \ep_{N}/2}\mbby_{1}(\gam)\le -\lam$, so that
\begin{align}
\sup_{|u_{1}|\ge (\ep_{N}/2)\sqrt{N}}\log \mbbz_{1,N}(u_{1})
& \lesssim N \bigg( \sup_{\theta}\left|\mbby_{1,N}(\theta)-\mbby_{1}(\gam)\right| +\sup_{\gam;\, |\gam-\gam_{0}|\ge \ep_{N}/2}\mbby_{1}(\gam)\bigg) \nn\\
& \le -\frac{\lam}{2}(\sqrt{N}\ep_{N})^{2} \lesssim -N^{1-2c}.
\nonumber
\end{align}
Therefore, recalling that $\mbbz_{1,N}(u_{1})=0$ outside the set $\mbbu_{1,N}:=\sqrt{N}(\Theta_{\gam}-\gam_{0})\subset\mbbr^{p_{\gam}}$ with $\Theta_{\gam}$ being bounded, we obtain the following estimate for some positive constants $C_{0}$ and $C_{1}$: on $G_{N}$,
\begin{align}
\int_{|u_{1}|\ge(\ep_{N}/2)\sqrt{N}} \mbbz_{1,N}(u_{1})\dif u_{1}
&\le e^{-C_{0}N^{1-2c}}\int_{\{|u_{1}|\ge (\ep_{N}/2)\sqrt{N}\}\cap\mbbu_{1,N}}\dif u_{1} \nn\\
&\lesssim e^{-C_{0}N^{1-2c}}\int_{|u_{1}|\le C_{1}\sqrt{N}}\dif u_{1} \lesssim N^{p_{\gam}/2}e^{-C_{0}N^{1-2c}} \to 0.
\label{hm:thm_bvM-p+4}
\end{align}
Hence \eqref{hm:thm_bvM-p+3} is obtained.

It follows from \eqref{hm:thm_bvM-p+4} that there exist $N_{1}\in\mbbn$ and $K>0$ such that for every $M_{1}>1$,
\begin{align}
& \sup_{N\ge N_{1}}\pr\bigg( G_{N} \cap\bigg\{ \int \mbbz_{1,N}(u_{1})\dif u_{1} > M_{1} \bigg\} \bigg) \nn\\
& \le \sup_{N\ge N_{1}}\pr\bigg( G_{N} \cap\bigg\{ \int_{|u_{1}|\ge K} \mbbz_{1,N}(u_{1})\dif u_{1} > \frac{M_{1}}{2} \bigg\} \bigg) \nn\\
&{}\qquad\qquad + \sup_{N\ge N_{1}}\pr\bigg( G_{N} \cap\bigg\{ \sup_{|u_{1}|\le K} \mbbz_{1,N}(u_{1})\gtrsim \frac{M_{1}}{2K^{p_{\gam}}} \bigg\} \bigg) \nn\\
&\le \sup_{N\ge N_{1}}\pr\bigg( \sup_{|u_{1}|\le K} \mbbz_{1,N}(u_{1})\gtrsim \frac{M_{1}}{2K^{p_{\gam}}} \bigg).
\nonumber
\end{align}
For every $K>0$ the sequence $\{\sup_{|u_{1}|\le K} \mbbz_{1,N}(u_{1})\}_{N}$ is tight in $\mbbr$, hence we can make the last upper bound arbitrarily small by taking a sufficiently large $M_{1}$.
This verifies \eqref{hm:thm_bvM-p+6}, and we are left to deduce \eqref{hm:thm_bvM-p++3}.

\bigskip

\noindent
\textit{Proof of \eqref{hm:thm_bvM-p++3}.}
Let
\begin{align*}
&a_{n-1}(\alpha) = a(X _ {(n - 1) h}, \alpha),\ 
c_{n-1}(\gamma) = c(X _ {(n - 1) h}, \gamma), \ 
\Delta^N_n X = X _ {nh} - X _ {(n-1)h},\\
&\epsilon_n(\theta) = \epsilon_{N,n}(\theta) := \frac{\Delta^N_nX - h a_{n-1}(\alpha)}{h ^ {1 / \beta}c_{n-1}(\gamma)}. 
\end{align*}
Let $z_{n}:=h^{-1/\beta}(J_{nh}-J_{(n-1)h})\sim J_{1}$ and $b_{n-1}(\theta):=c^{-1}_{n-1}(\gam)\{a_{n-1}(\al_{0})-a_{n-1}(\al)\}$.
Below we will repeatedly make use of several statements in \cite[Sections 6.2 and 6.3]{masuda16.sqmle}, hence at this stage it should be noted that, by the localization procedure \cite[Section 6.1 and the references therein]{masuda16.sqmle}, without loss of generality we may and do suppose that
\begin{align}\nn%\label{eq:known}
\mathbb{E}\left(\sup_{t\le T}|X_t|^q\right)\lesssim 1,\quad 
\sup_{t\in [s,s+h]\cap [0,T]}\mathbb{E}\left(|X_t-X_s|^q|\mathcal{F}_s\right)\lesssim h(1+|X_s|^C)
\end{align}
for any $q\ge 2$ and $s\in [0,T]$.
%Since we are only concerned with weak property, by the localization (removal of big jumps) argument with the global Lipschitz condition we may and do proceed as if the following moment bounds hold true:
%for any $q\ge 2$ and $s\in [0,T]$ we have
%\begin{align}\nn%\label{eq:known}
%\mathbb{E}\left(\sup_{t\le T}|X_t|^q\right)\lesssim 1,\quad 
%\sup_{t\in [s,s+h]\cap [0,T]}\mathbb{E}\left(|X_t-X_s|^q|\mathcal{F}_s\right)\lesssim h(1+|X_s|^C)	
%\end{align}
%for some $C>0$.
This fact in particular implies that 
\begin{equation}
\sup_{t\in [s,s+h]\cap [0,T]}\mathbb{E}(|X_t-X_s|)\lesssim \sqrt{h}.
\label{hm:thm_bvM-p+8}
\end{equation}
For convenience, for a sequence of random functions $\{f_{N}(\cdot)\}$ on $\overline{\Theta}$ and a positive sequence $(b_{N})$ we will write $f_{N}(\theta)=O^{\ast}_{p}(b_{N})$ if $\sup_{\theta}|f_{N}(\theta)|=O_{p}(b_{N})$.

\medskip

We can write $\mbby_{1,N}(\theta)$ as
\begin{align}
\mbby_{1,N}(\theta) &= \frac{1}{N}\sum_{n=1}^{N}\log\frac{c_{n-1}(\gam_{0})}{c_{n-1}(\gam)}
+\frac{1}{N}\sum_{n=1}^{N}\log\phi_{\beta}\left(\ep_{n}(\theta)\right)
-\frac{1}{N}\sum_{n=1}^{N}\log\phi_{\beta}\left(\ep_{n}(\al,\gam_{0})\right).
\label{hm:thm_bvM-p+5}
\end{align}
The function $y\mapsto\log\phi_{\beta}(y)$ fulfills the conditions on $\eta$ in Lemmas 6.2 and 6.3 of \cite{masuda16.sqmle}.
Applying the two lemmas to the second and third terms in the right-hand side of \eqref{hm:thm_bvM-p+5}, we can deduce that
\begin{align}
\mbby_{1,N}(\theta) &= \frac{1}{N}\sum_{n=1}^{N}\log\frac{c_{n-1}(\gam_{0})}{c_{n-1}(\gam)}
+\frac{1}{N}\sum_{n=1}^{N}\E\bigg\{\log\phi_{\beta}\bigg(
\frac{c_{n-1}(\gam_{0})}{c_{n-1}(\gam)}z_{n}+h^{1-1/\beta}b_{n-1}(\theta)\bigg)\bigg|\mcf_{(n-1)h}\bigg\} \nn\\
&{}\qquad -\frac{1}{N}\sum_{n=1}^{N}\E\bigg\{\log\phi_{\beta}\bigg(
z_{n}+h^{1-1/\beta}b_{n-1}(\al,\gam_{0})\bigg)\bigg|\mcf_{(n-1)h}\bigg\} 
+O^{\ast}_{p}\bigg(\frac{1}{\sqrt{N}} \vee h^{2-1/\beta}\bigg).
\nn
\end{align}
Proceeding as in the second equality in Eq.(6.19) of \cite{masuda16.sqmle}, we obtain
\begin{align}
\mbby_{1,N}(\theta) &= \frac{1}{N}\sum_{n=1}^{N}\bigg[
\log\frac{c_{n-1}(\gam_{0})}{c_{n-1}(\gam)}
+\int\bigg\{
\log\phi_{\beta}\bigg(\frac{c_{n-1}(\gam_{0})}{c_{n-1}(\gam)}z\bigg) - \log\phi_{\beta}(z)\bigg\}\phi_{\beta}(z)dz\bigg]
\nn\\
&{}\qquad +O^{\ast}_{p}\bigg(h^{1-1/\beta} \vee \frac{1}{\sqrt{N}} \vee h^{2-1/\beta}\bigg).
\label{hm:thm_bvM-p+6'}
\end{align}
Using the estimate \eqref{hm:thm_bvM-p+8} combined with the inequality given in the proof of \cite[Lemma 6.4]{masuda16.sqmle}, we can deduce from \eqref{hm:thm_bvM-p+6'} that
\begin{align}
\sup_{\theta}|\mbby_{1,N}(\theta)-\mbby_{1}(\gam)|
&\le O_{p}\bigg(\sqrt{h} \vee h^{1-1/\beta} \vee \frac{1}{\sqrt{N}} \vee h^{2-1/\beta}\bigg) \nn\\
&= O_{p}(h^{1-1/\beta})=O_{p}\left(N^{-(1-1/\beta)}\right).
\nonumber
\end{align}
Hence $(\sqrt{N})^{q}\sup_{\theta}|\mbby_{1,N}(\theta)-\mbby_{1}(\gam)| \cip 0$ holds for $q\in\left(0, 2(1-1/\beta)\right)$.

As for $\mbby_{2,N}$, following a similar line to the case of $\mbby_{1,N}$ we can derive
\begin{align}
\mbby_{2,N}(\al) &= \frac{1}{2N}\sum_{n=1}^{N}b^{2}_{n-1}(\al,\gam_{0})
\Big(\p_{y}^{2}\log\phi_{\beta}(y)\Big)\Big|_{y=\ep_{n}(\tz)}
+ O^{\ast}_{p}\bigg(\frac{1}{\sqrt{N}h^{1-1/\beta}} \vee h^{1-1/\beta}\bigg).
\nonumber
\end{align}
Here, we also made use of \cite[Corollary 6.6]{masuda16.sqmle} (the function $y\mapsto\p_{y}\log\phi_{\beta}(y)$ is odd) and the arguments in \cite[Section 6.3.2]{masuda16.sqmle}.
Then it is not difficult to arrive at $\sup_{\al}|\mbby_{2,N}(\al)-\mbby_{2}(\al)|=O_{p}(N^{-s})$ for any sufficiently small $s>0$. In view of \eqref{hm:thm_bvM-p+7} we conclude that $(\sqrt{N}h^{1-1/\beta})^{q}\sup_{\al}|\mbby_{2,N}(\al)-\mbby_{2}(\al)| \cip 0$ for $q>0$ small enough. The proof of \eqref{hm:thm_bvM-p++3} is thus complete.

%%%%%
%%%%%

\section{Proof for the local consistency of the Metropolis-Hastings algorithm}
\label{app:lc}

We prove Proposition \ref{prop:mtk} in this section. 
Without loss of generality, we can assume $s_N\rightarrow s$ in probability for some random variable $s$. 
For notational simplicity, we write $\bar{F}^N$ for the rescaled version of the function or measure $F^N$ by $u\mapsto \theta_0+D_N^{-1}u$, and write 
$$
\bar{\nu}_N(\dif u,\dif v|x^N)=\bar{\Pi}(\dif u|x^N)\bar{Q}_N(u,\dif v|x^N),\qquad
\bar{\mu}_N(\dif u,\dif v|x^N)=\bar{\Pi}(\dif u|x^N)\bar{P}_N(u,\dif v|x^N)
$$
and 
$$
\nu(\dif u,\dif v|s)=\Pi(\dif u|s)Q(u,\dif v|s),\qquad
\mu(\dif u,\dif v|s)=\Pi(\dif u|s)P(u,\dif v|s). 
$$
The equation (\ref{BvMMH}) becomes 
$$\|\bar{\nu}_N(\cdot|x^N)-\nu(\cdot|s_N)\|_{\mathrm{TV}}\rightarrow 0.$$
By Lemmas 2 and 3 of \cite{Kamatani10} the following convergence is sufficient for local consistency:
$$
\delta_N:=\|\bar{\mu}_N(\cdot|x^N)-\mu(\cdot|s)\|_{\mathrm{TV}}=o_p(1). 
$$
First, we prove the convergence of $\delta_N':=\|\mu(\cdot|s_N)-\mu(\cdot|s)\|_{\mathrm{TV}}$.
By triangular inequality, $\delta_N'$ is dominated above by the sum of 
\begin{align*}
\delta_{N,1}':=&\|\nu(\cdot|s)A(\cdot|s)-\nu(\cdot|s_N)A(\cdot|s_N)\|_{\mathrm{TV}},\\ 	
\delta_{N,2}':=&\|\Pi(\dif u|s)R(u|s)\delta_u(\dif v)-\Pi(\dif u|s_N)R(u|s_N)\delta_u(\dif v)\|_{\mathrm{TV}}. 
\end{align*}
For the former, by Assumption \ref{as:kernel}, 
\begin{align*}
\delta_{N,1}'=2\int\left(\pi(u|s)q(u,v|s)A(u,v|s)-\pi(u|s_N)q(u,v|s_N)A(u,v|s_N)\right)^+\dif u\dif v=o_p(1)
\end{align*}
by the dominated convergence theorem, where $x^+=\max\{0,x\}$. We can prove $\delta_{N,3}':=\|\nu(\cdot|s)-\nu(\cdot|s_N)\|_{\mathrm{TV}}=o_p(1)$
in the same way. 
On the other hand, by triangular inequality, 
\begin{align*}
	\delta_{N,2}'&=\|\Pi(\dif u|s)R(u|s)-\Pi(\dif u|s_N)R(u|s_N)\|_{\mathrm{TV}}\\
&=\left\|\left\{\nu(\cdot\times \mathbb{R}^p|s)-\int\nu(\cdot\times\dif v|s)A(\cdot,v|s)\right\}-\left\{\nu(\cdot\times \mathbb{R}^p|s_N)-\int\nu(\cdot\times v|s_N)A(\cdot,v|s_N)\right\}\right\|_{\mathrm{TV}}\\
&\le \delta_{N,3}'+\delta_{N,1}'=o_p(1). 
\end{align*}
Hence, $\delta_N'\le\delta_{N,1}'+\delta_{N,2}'\rightarrow 0$ in probability. 

Next, we prove the convergence of $\delta_N'':=\|\bar{\mu}_N(\cdot|x^N)-\mu(\cdot|s_N)\|_{\mathrm{TV}}$.
By the same argument as above, it is sufficient to show the convergence of 
$\delta_{N,1}'':=\|\bar{\nu}_N(\cdot|x^N)\bar{A}_N(\cdot|x^N)-\nu(\cdot|s_N)A(\cdot|s_N)\|_{\mathrm{TV}}$
and $\delta_{N,3}'':=\|\bar{\nu}_N(\cdot|x^N)-\nu(\cdot|s_N)\|_{\mathrm{TV}}$. 
We only proves the former. 
By triangular inequality, 
\begin{align*}
\delta_{N,1}''&\le \left\|\left\{\bar{\nu}_N(\cdot|x^N)-\nu(\cdot|s_N)\right\}\bar{A}_N(\cdot|x^N)\right\|_{\mathrm{TV}}+\left\|\nu(\cdot|s_N)\left(\bar{A}_N(\cdot|x^N)-A(\cdot|s_N)\right)\right\|_{\mathrm{TV}}\\
&\le
\left\|\bar{\nu}_N(\cdot|x^N)-\nu(\cdot|s_N)\right\|_{\mathrm{TV}}+\int |\bar{A}_N(u,v|x^N)-A(u,v|s_N)|\nu(\dif u,\dif v|s_N).
\end{align*}
The first term converges to $0$ by assumption. 
Since $\nu(\cdot|s_N)$ converges to $\nu(\cdot|s)$, 
and $s_N$ is tight, 
for any $\epsilon>0$, we can choose a compact set $K_p\subset\mathbb{R}^p\times\mathbb{R}^p$ and $K_q\subset\mathcal{S}$ so that 
$$
\limsup_{N\rightarrow\infty}\mathbf{E}_{\theta_0}^{(N)}\left[\nu(K_p^c|s_N)\right]<\epsilon/2,\qquad \limsup_{N\rightarrow\infty}\mathbf{P}_{\theta_0}^{(N)}(s_N\in K_q^c)<\epsilon/2. 
$$
Thus, 
\begin{align*}
&\mathbf{E}_{\theta_0}^{(N)}\left[\int |\bar{A}_N(u,v|x^N)-A(u,v|s_N)|\nu(\dif u,\dif v|s_N)\right]\\
&\le \epsilon + \mathbf{E}_{\theta_0}^{(N)}\left[\int_{K_p} |\bar{A}_N(u,v|x^N)-A(u,v|s_N)|\nu(\dif u,\dif v|s_N), s_N\in K_q\right] + o(1)\\
&\le \epsilon + c\int_{K_p} \mathbf{E}_{\theta_0}^{(N)}\left[|\bar{A}_N(u,v|x^N)-A(u,v|s_N)|\right]\dif u\dif v+ o(1)=\epsilon +o(1), 
\end{align*}
where $c=\sup_{s\in K_q}\sup_{u,v}\pi(u|s)q(u,v|s)$ is the upper bound of the probability density function of $\nu(\dif u,\dif v|s_N)$ when $s_N\in K_q$. 
This completes the proof of $\delta_N''\rightarrow 0$ in probability, and hence 
$\delta_N\le \delta_N'+\delta_N''=o_p(1)$.

\section{Convergence of the acceptance ratio}

\subsection{Setting and notation}\label{hm:sec_conv-ar}

We keep using some notation introduced in Section \ref{app:bvm}. 
We consider 
an extended probability space
\begin{align}\label{eq:oa}
	\tilde{\Omega} = \Omega \times A,\ 
	\tilde{\mathcal{F}} = \mathcal{F} \otimes \mathcal{A},\ 
	\tilde{\mathbb{P}} = \mathbb{P}(\dif \omega)\mathbb{Q}_\omega(\dif a)
\end{align}
where $(A,\mathcal{A})$ is a measurable space, and $\mathbb{Q}_\omega(\dif a)$
is a probability transition kernel from $\Omega$ to $A$. We consider a stochastic process
\begin{align*}
	\tilde{X}_t(\omega, a) = X _ t(\omega)=\omega_t.
\end{align*}
Fix $u \in \mathbb{R} ^ d$. Set 
$\theta_N=(\alpha_N,\gamma_N)=\theta_0+D_N^{-1}u$. 
We now consider random variables $V_n,\ n=1,\ldots, N$, which correspond to the pseudo-data generated in Algorithm \ref{alg:1} from parameter 
$\theta=\theta_N$.  Thus, for each $\omega\in\Omega$, the random variables $V_n(\omega,\cdot):A\rightarrow\mathbb{R}_+\ (n = 1, \ldots, N)$
are independent, and 
\begin{align*}
&\tilde{\mathbb{P}} \left( V _ n \in \dif v |\omega \right) = \mathbb{Q}_\omega(V_n\in\dif v)=F_\beta(\dif v|\epsilon_n(\theta_N)).
\end{align*}
%where 
%\begin{align*}
%&a_{n-1}(\alpha) = a(X _ {(n - 1) h}, \alpha),\ 
%c_{n-1}(\gamma) = c(X _ {(n - 1) h}, \gamma), 
%\Delta^N_n X = X _ {nh} - X _ {(n-1)h},\\
%&\epsilon_n(\theta) = \epsilon_{N,n}(\theta) := \frac{\Delta^N_nX - h a_{n-1}(\alpha)}{h ^ {1 / \beta}c_{n-1}(\gamma)}. 
%\end{align*}
The log likelihood and the augmented-data (pseudo $+$ observed data) log likelihood are
$$
\mathbb{H}_N(\theta)=-\sum_{n=1}^N\log c_{n-1}(\gamma)+\sum_{n=1}^N\log\phi_\beta(\epsilon_n(\theta)),\ 
\mathbb{H}^\dagger_N(\theta)=
-\sum_{n=1}^N\log c_{n-1}(\gamma)-\frac{1}{2}\sum_{n=1}^N\frac{\epsilon_n(\theta)^2}{V_n}.
$$
Here we omit terms that do not include $\theta$. 
We define the Fisher information matrix for the augmented-data model by 
\begin{align*}
I ^ \dagger(\theta _ 0)	:=
\mathrm{diag}\left(
C _ \alpha ^ \dagger (\beta) \Sigma _ {T, \alpha}(\theta _ 0), 
C _ \gamma ^ \dagger (\beta) \Sigma _ {T, \gamma}(\gamma _ 0)
\right), 
\end{align*}
with
\begin{align*}
	C_\alpha^\dagger(\beta):=\int_{\mathbb{R}_+}\frac{1}{v}F_\beta(\dif v), \quad C^\dagger_\gamma(\beta):=2
\end{align*}
and that for the pseudo-data model by 
\begin{align*}
I ^ *(\theta _ 0)	:=I ^ \dagger(\theta _ 0)-I(\theta _ 0)=
\mathrm{diag}\left(
C _ \alpha ^ * (\beta) \Sigma _ {T, \alpha}(\theta _ 0), 
C _ \gamma ^ * (\beta) \Sigma _ {T, \gamma}(\gamma _ 0)
\right), 
\end{align*} 
with
$$
C_\alpha^*(\beta):=C_\alpha^\dagger(\beta)-C_\alpha(\beta), \quad C^*_\gamma(\beta):=C^\dagger_\gamma(\beta)-C_\gamma(\beta).
$$

%By the localization procedure \cite[Section 6.1 and the references therein]{masuda16.sqmle}, without loss of generality we may suppose that for any $q\ge 2$ and $s\in [0,T]$
%\begin{align}\nn%\label{eq:known}
%\mathbb{E}\left(\sup_{t\le T}|X_t|^q\right)\lesssim 1,\quad 
%\sup_{t\in [s,s+h]\cap [0,T]}\mathbb{E}\left(|X_t-X_s|^q|\mathcal{F}_s\right)\lesssim h(1+|X_s|^C).
%\end{align}
%%Since we are only concerned with weak property, by the localization (removal of big jumps) argument with the global Lipschitz condition we may and do proceed as if the following moment bounds hold true:
%%for any $q\ge 2$ and $s\in [0,T]$ we have
%%\begin{align}\nn%\label{eq:known}
%%\mathbb{E}\left(\sup_{t\le T}|X_t|^q\right)\lesssim 1,\quad 
%%\sup_{t\in [s,s+h]\cap [0,T]}\mathbb{E}\left(|X_t-X_s|^q|\mathcal{F}_s\right)\lesssim h(1+|X_s|^C)	
%%\end{align}
%%for some $C>0$.

In the next section, we will use the following law of large numbers. 

\begin{proposition}
\label{lem:lln}[Proposition 6.5 of \cite{masuda16.sqmle}]
If $\eta(x)\in \mathcal{C}^1(\mathbb{R})$ and $\pi(x,\theta)\in\mathcal{C}^1(\mathbb{R}\times\Theta)$ satisfy $\sup_{x\in\mathbb{R}}|\eta(x)|+|\eta'(x)|\lesssim 1$ and $\sup_\theta|\pi(x,\theta)|+|\partial_\theta\pi(x,\theta)|\lesssim 1+|x|^C$ for some $C>0$, then
\begin{align}
& N^{-1}\sum_{n=1}^N\eta(\epsilon_n(\theta_N))\pi(X_{(n-1)h},\theta_N)=O_p(1), \nn\\
& N^{-1}\sum_{n=1}^N\eta(\epsilon_n(\theta_N))\pi(X_{(n-1)h},\theta_N)= \frac{1}{T}\int\eta(x)\phi_\beta(x)\dif x\int_0^T\pi(X_t,\theta_0)\dif t+o_p(1).
\label{eq:lln}
\end{align}
%If $\eta(x)$ is bounded, and $\sup_\theta|\pi(x,\theta)|\lesssim 1+|x|^C$ for some $C>0$, then
%\begin{align*}
%	N^{-1}\sum_{n=1}^N\eta(\epsilon_n(\theta_N))\pi(X_{(n-1)h},\theta_N)=O_p(1).
%\end{align*}	
%Moreover, 
%if  $\eta(x)\in \mathcal{C}^1(\mathbb{R})$ and $\pi(x,\theta)\in\mathcal{C}^1(\mathbb{R}\times\Theta)$ satisfy $\sup_{x\in\mathbb{R}}|\eta(x)|+|\eta'(x)|\lesssim 1$ and $\sup_\theta|\pi(x,\theta)|+|\partial_\theta\pi(x,\theta)|\lesssim 1+|x|^C$ for some $C>0$, then 
%\begin{align}\nn%\label{eq:lln}
%	N^{-1}\sum_{n=1}^N\eta(\epsilon_n(\theta_N))\pi(X_{(n-1)h},\theta_N)= \frac{1}{T}\int\eta(x)\phi_\beta(x)\dif x\int_0^T\pi(X_t,\theta_0)\dif t+o_p(1). 
%\end{align}
\end{proposition}

We also use the following fact: Let $\beta\in (1,2)$. If $\pi(x,\theta)$ satisfies the condition in Proposition \ref{lem:lln}, then
\begin{align}\label{eq:g}
	\frac{1}{Nh^{1-1/\beta}}\sum_{n=1}^N g_\beta(\epsilon_n(\theta_N))\pi(X_{(n-1)h},\theta_N)=O_p((\sqrt{N}h^{1-1/\beta})^{-1}). 
\end{align}
This convergence comes from Corollary 6.6. of \cite{masuda16.sqmle} together with the estimate
\begin{align*}
	|g_\beta(\epsilon_n(\theta_N))-g_\beta(\epsilon_n(\alpha_0,\gamma_N))|\lesssim |\epsilon_n(\theta_N)-\epsilon_n(\alpha_0,\gamma_N))|\lesssim \frac{1}{\sqrt{N}h^{1-1/\beta}}(1+|X_{(n-1)h}|^C). 
\end{align*}

\subsection{Some properties of $F_\beta$}

Let 
\begin{align*}
	g_\beta(x):=\frac{\partial}{\partial x} \log \phi _ \beta(x),\ 
	h_\beta(x):=\frac{\partial^2}{\partial x^2}\log\phi_\beta(x)-\frac{1}{x}\frac{\partial}{\partial x}\log\phi_\beta(x). 
\end{align*}
Recall that by the property of stable law (see pp.88--89 of \cite{MR1739520}), we
have $\phi_\beta(x)\sim c|x|^{-(\beta +1)}$ as  $|x|\rightarrow\infty$ for some $c>0$.  
Moreover, the probability density function 
$f_\beta(v)$ of $F_\beta$ is bounded above, and as $v\rightarrow+\infty$, we have  $f_\beta(v)\sim c|v|^{-\beta/2-1}$ for some $c>0$, 
and as $v\rightarrow 0$ the density of the positive stable distribution $f_\beta(v)$ converges to $0$ exponentially fast. 
Thus $\int v^{-k}F_\beta(\dif v|x)$ is continuous at $x=0$ for $k\ge 0$. 
Moreover, we have the following identities and an estimate. 

\begin{lemma}
Fix $\beta \in [1,2)$. 
\begin{enumerate}
\item 
\begin{equation}\label{eq:gh}
g_\beta(x)=-\int\frac{x}{v}F_\beta(\dif v|x),\ 
h_\beta(x)=\int\left(\frac{x}{v}+g_\beta(x)\right)^2F_\beta(\dif v|x). 
\end{equation}
	\item 
\begin{equation}\label{eq:identity}
\int x g_\beta(x)\phi_\beta(x)\dif x=-1,\ \int h_\beta(x)\phi_\beta(x)\dif x=C_\alpha^*(\beta), \int x^2h_\beta(x)\phi_\beta(x)\dif x=C_\gamma^*(\beta).	
\end{equation}
 \item 
For any $k\ge 0$, 
\begin{align}\label{eq:integrability}
\int_0^\infty v^{-k}F_\beta(\dif v|x)\lesssim |x|^{-2k}.
\end{align}
\end{enumerate}
\end{lemma}

\begin{proof}
The expression of $g_\beta(x)$ can be obtained via simple interchange of the derivative and the integral in the equation (\ref{eq:fb}). 
For the expression of $h_\beta(x)$, we have
\begin{align*}
	h_\beta(x)&=\frac{\phi_\beta''(x)}{\phi_\beta(x)}-g_\beta(x)^2-\frac{1}{x}g_\beta(x). 
\end{align*}
By (\ref{eq:pb}), the second derivative of $\phi_\beta(x)$ is
\begin{align*}
\frac{\partial^2}{\partial x^2}\int_{\mathbb{R}_+} \frac{1}{\sqrt{2\pi v}}\exp\left(-\frac{x^2}{2v}\right)F_\beta(\dif v)
&=
\int_{\mathbb{R}_+} \left(\frac{x^2}{v^2}-\frac{1}{v}\right)\frac{1}{\sqrt{2\pi v}}\exp\left(-\frac{x^2}{2v}\right)F_\beta(\dif v)\\
	&=\phi_\beta(x)\int_{\mathbb{R}_+} \left(\frac{x^2}{v^2}-\frac{1}{v}\right)F_\beta(\dif v|x).
\end{align*}
This equation yields the expression in (\ref{eq:gh}) by using the identity $\int v^{-1}F_\beta(\dif v|x)=-g_\beta(x)/x$. 

Next, we prove identities (\ref{eq:identity}). Applying the change of variable $u=x/\sqrt{v}$, we have
\begin{align*}
	\int x g_\beta(x)\phi_\beta(x)\dif x&=-\int x\left\{\frac{x}{v}F_\beta(\dif v|x)\right\}\phi_\beta(x)\dif x\\
	&=-\int\frac{x^2}{v}\frac{1}{\sqrt{2\pi v}}\exp\left(-\frac{x^2}{2v}\right)\dif xF_\beta(\dif v)=
-\int u^2\phi(u)\dif u=-1
\end{align*}
where $\phi(u)$ is the probability density function of the standard normal distribution. 
In the same way, by the change of variable, we obtain
$$
\int \left(\frac{x}{v}\right)^2F_\beta(\dif v|x)\phi_\beta(x)\dif x=C_\alpha^\dagger(\beta),\ 
\int x^2\left(\frac{x}{v}\right)^2F_\beta(\dif v|x)\phi_\beta(x)\dif x=3. 
$$
Then we have
$$\int h_\beta(x)\phi_\beta(x)\dif x=\int \left(\frac{x}{v}\right)^2F_\beta(\dif v|x)\phi_\beta(x)\dif x-\int g_\beta(x)^2\phi_\beta(x)\dif x= C_\alpha^*(\beta)$$
and 
\begin{align*}
\int x^2h_\beta(x)\phi_\beta(x)\dif x&=\int x^2\left(\frac{x}{v}\right)^2F_\beta(\dif v|x)\phi_\beta(x)\dif x-\int x^2g_\beta(x)^2\phi_\beta(x)\dif x\\
&=3-\int x^2g_\beta(x)^2\phi_\beta(x)\dif x=2-\int (1+xg_\beta(x))^2\phi_\beta(x)\dif x=C_\gamma^*(\beta). 
\end{align*}
Finally we check (\ref{eq:integrability}). By the property of stable law, 
\begin{align*}
	\int_0^\infty\frac{x^{2k}}{v^k}F_\beta(\dif v|x)&=\phi_\beta(x)^{-1}\int_0^\infty\frac{x^{2k}}{v^k}\frac{1}{\sqrt{2\pi v}}\exp\left(-\frac{x^2}{2v}\right)F_\beta(\dif v)\\
	&\lesssim\int_0^\infty\left(\frac{x^2}{v}\right)^{(2k+\beta+1)/2}\exp\left(-\frac{x^2}{2v}\right)\frac{\dif v}{v}\\
	&= \int_0^\infty\left(\frac{1}{u}\right)^{(2k+\beta+1)/2}\exp\left(-\frac{1}{2u}\right)\frac{\dif u}{u}<\infty
\end{align*}
by the change of variable $u =v/x^2$. 
Thus, the claim follows. 
\end{proof}

\subsection{Estimates for the likelihood functions}

Let $\Delta_{N}^\dagger(\theta_N)=D_N^{-1}\partial_\theta\mathbb{H}^\dagger_N(\theta_N)$
and set $\Delta_{N}^*(\theta_N)=\Delta_{N}^\dagger(\theta_N)-\Delta_{N}(\theta_N)$. 

\begin{lemma}
	\begin{equation}\label{eq:clan_centering}
	\|\mathcal{L}(\Delta_{N}^*(\theta_N)|\omega)- N_k(0, I^*(\theta_0))\|_{\mathrm{BL}}\rightarrow 0
	\end{equation}
	in $\mathbb{P}$-probability. 
\end{lemma}

\begin{proof}
	Observe that
	\begin{align*}
		\partial_\alpha\epsilon_n(\theta)=-h^{1-1/\beta}\frac{\partial_\alpha a_{n-1}(\alpha)}{c_{n-1}(\gamma)},\ 
		\partial_\gamma\epsilon_n(\theta)=-\epsilon_n(\theta)\frac{\partial_\gamma c_{n-1}(\gamma)}{c_{n-1}(\gamma)}. 
	\end{align*}
	By this fact, we have
	\begin{align*}
		\Delta_{N}^\dagger(\theta_N)=
		N^{-1/2}\sum_{n=1}^N
		\begin{pmatrix}
			\frac{\epsilon_n(\theta_N)}{V_n}	
			\frac{\partial_\alpha a_{n-1}(\alpha_N)}{c_{n-1}(\gamma_N)}\\
			\left\{
				\frac{\epsilon_n(\theta_N)^2}{V_n} - 1
			\right\}
			\frac{\partial_\gamma c_{n-1}(\gamma_N)}{c_{n-1}(\gamma_N)}
		\end{pmatrix},
		\end{align*}
		and hence 
	\begin{align*}
		\Delta_{N}^*(\theta_N)=
		N^{-1/2}\sum_{n=1}^N
		\begin{pmatrix}
			\left\{
			\frac{\epsilon_n(\theta_N)}{V_n}	
			+g_\beta(\epsilon_n(\theta_N))\right\}
			\frac{\partial_\alpha a_{n-1}(\alpha_N)}{c_{n-1}(\gamma_N)}\\
			\left\{
				\frac{\epsilon_n(\theta_N)}{V_n}+g_\beta(\epsilon_n(\theta_N))
			\right\}\epsilon_n(\theta_N)
			\frac{\partial_\gamma c_{n-1}(\gamma_N)}{c_{n-1}(\gamma_N)} 
		\end{pmatrix}
		=:N^{-1/2}\sum_{n=1}^N\xi_n.
		\end{align*}
		Recall the definition of the extended probability space (\ref{eq:oa}). In this setting, the pseudo-data variables $V_1,\ldots, V_N$ are independent conditioned on $\omega\in\Omega$ and generated from $\mathbb{Q}_\omega$. 
		Fix $\omega\in\Omega$. Since $\int (x/v+g_\beta(x))F_\beta(\dif v|x)=0$,  the random variable $\Delta_{N}^*(\theta_N)(\omega,\cdot):A\rightarrow\mathbb{R}^p$ is a sum of independent variables  $\xi_n(\omega,\cdot)$ in $\mathbb{Q}_\omega$. 
		By the expression of $h_\beta$, the covariance matrix of $\xi_n$ conditioned on $\omega$ becomes 
		\begin{align*}
		\tilde{\mathbb{E}}[\xi_n^{\otimes 2}|\omega]=
		\begin{pmatrix}
			h_\beta(\epsilon_n(\theta_N))
			\left(\frac{\partial_\alpha a_{n-1}(\alpha_N)}{c_{n-1}(\gamma_N)}\right)^{\otimes 2}& 
			h_\beta(\epsilon_n(\theta_N))\epsilon_n(\theta_N)
			\left(\frac{\partial_\alpha a_{n-1}(\alpha_N)}{c_{n-1}(\gamma_N)}\right)\otimes \left(\frac{\partial_\gamma c_{n-1}(\gamma_N)}{c_{n-1}(\gamma_N)}\right)\\
			Sym.
			 &
			 			h_\beta(\epsilon_n(\theta_N))\epsilon_n(\theta_N)^2
			 \left(\frac{\partial_\gamma c_{n-1}(\gamma_N)}{c_{n-1}(\gamma_N)}\right)^{\otimes 2}
		\end{pmatrix}.
		\end{align*}
		This conditional covariance can be written as $\tilde{\mathbb{E}}[\xi_n^{\otimes 2}|\omega]=\eta(\epsilon_n(\theta_N))[\pi(X_{(n-1)h},\theta_N)]$ where 
		$$
		\eta(x)=h_\beta(x)\begin{pmatrix}1&x\\ x&x^2\end{pmatrix},\ \pi(x,\theta)=\left[\frac{\partial_\alpha a_{n-1}(\alpha_N)}{c_{n-1}(\gamma_N)},\frac{\partial_\gamma c_{n-1}(\gamma_N)}{c_{n-1}(\gamma_N)}\right]^{\otimes 2}. 
		$$
		By (\ref{eq:integrability}) and Assumption \ref{hm:A1}, $\eta(x)$ and $\pi(x,\theta)$ satisfy the condition of Proposition \ref{lem:lln}. On the other hand,  for the forth moment on $\xi_n$, we have
		\begin{align*}
			\tilde{\mathbb{E}}[|\xi_n|^4|\omega]^{1/4}\lesssim &
		\left\{\int \left|\frac{\epsilon_n(\theta_N)}{v}\right|^4F_\beta(\dif v|\epsilon_n(\theta_N))\right\}^{1/4}\left|\frac{\partial_\alpha a_{n-1}(\alpha_N)}{c_{n-1}(\gamma_N)}\right|\\
		&+
				\left\{\int \left|\frac{\epsilon_n(\theta_N)^2}{v}\right|^4F_\beta(\dif v|\epsilon_n(\theta_N))\right\}^{1/4}\left|\frac{\partial_\gamma c_{n-1}(\alpha_N)}{c_{n-1}(\gamma_N)}\right|\\
		\\
		\lesssim& 1+|X_{(n-1)h}|^C
		\end{align*}
		for some $C>0$ by (\ref{eq:integrability}) and Assumption \ref{hm:A1} together with 
		the Minkowski and Jensen inequalities. 
		Therefore by Proposition \ref{lem:lln}, 
		\begin{align*}
			\tilde{\mathbb{E}}[(\Delta_{N}^*(\theta_N))^{\otimes 2}|\omega]=N^{-1}\sum_{n=1}^N\tilde{\mathbb{E}}[\xi_n^{\otimes 2}|\omega]\rightarrow I^*(\theta_0),\ 
			N^{-2}\sum_{n=1}^N\tilde{\mathbb{E}}[|\xi_n|^4|\omega]\rightarrow 0
		\end{align*}
		in $\tilde{\mathbb{P}}$-probability. For any infinite elements $\mathbb{N}'\subset \mathbb{N}$ there exists a further subsequence $\mathbb{N}''=\{N''\}\subset \mathbb{N}'$
		such that the above convergence satisfies in almost surely in $\omega\in\Omega$. Therefore, by the central limit theorem, 
		$$
		\mathcal{L}(\Delta_{N''}^*(\theta_N)|\omega)\rightarrow N_k(0, I^*(\theta_0))
		$$
		almost surely. Thus, the claim follows. 
\end{proof}

\begin{lemma}
	\begin{equation}\label{eq:clan_fishermat}
	-D_N^{-1}\partial^2_\theta\mathbb{H}^\dagger_N(\theta_N)D_N^{-1}\rightarrow I^\dagger(\theta_0)		
	\end{equation}
	in $\tilde{\mathbb{P}}$-probability, and 
	\begin{equation}\label{eq:clan_remainder}
	\sup_\Theta\max_{k=1,\ldots, p}|\partial_{\theta_k}D_N^{-1}\partial^2_\theta\mathbb{H}^\dagger_N(\theta)D_N^{-1}|=O_p(1). 
	\end{equation}
\end{lemma}

\begin{proof}
By calculation, 
\begin{align}\label{eq:aa}
(Nh^{2(1-\beta)})^{-1}\partial_\alpha^2 \mathbb{H}_N^\dagger (\theta_N)	
=\frac{1}{Nh^{1-1/\beta}}\sum_{n=1}^N\frac{\epsilon_n(\theta_N)}{V_n}\left(\frac{\partial_\alpha^2 a_{n-1}(\alpha_N)}{c_{n-1}(\gamma_N)}\right)-\frac{1}{N}\sum_{n=1}^N\frac{1}{V_n}\left(\frac{\partial_\alpha a_{n-1}(\alpha_N)}{c_{n-1}(\gamma_N)}\right)^{\otimes 2}. 
\end{align}
The first term of the right-hand side of (\ref{eq:aa}) is 
\begin{align*}
-\frac{1}{Nh^{1-1/\beta}}\sum_{n=1}^Ng_\beta(\epsilon_n(\theta_N))\left(\frac{\partial_\alpha^2 a_{n-1}(\alpha_N)}{c_{n-1}(\gamma_N)}\right)
+\frac{1}{Nh^{1-1/\beta}}\sum_{n=1}^N\left(\frac{\epsilon_n(\theta_N)}{V_n} + g_\beta(\epsilon_n(\theta_N))\right)\left(\frac{\partial_\alpha^2 a_{n-1}(\alpha_N)}{c_{n-1}(\gamma_N)}\right). 
\end{align*}
Then, the first term is negligible by Proposition \ref{eq:lln} for $\beta=1$ and by (\ref{eq:g}) for $\beta\in (1,2)$, and the second term is also negligible since it is a sum of independent variables conditioned on $\omega$, and its variance is $o_p(1)$. 
Also, the second term of (\ref{eq:aa}) is 
\begin{align*}
\frac{1}{N}\sum_{n=1}^N\left(\frac{g_\beta(\epsilon_n(\theta_N))}{\epsilon_n(\theta_N)}\right)\left(\frac{\partial_\alpha a_{n-1}(\alpha_N)}{c_{n-1}(\gamma_N)}\right)^{\otimes 2}
-
\frac{1}{N}\sum_{n=1}^N\left(\frac{1}{V_n}+\frac{g_\beta(\epsilon_n(\theta_N))}{\epsilon_n(\theta_N)}\right)\left(\frac{\partial_\alpha a_{n-1}(\alpha_N)}{c_{n-1}(\gamma_N)}\right)^{\otimes 2},
\end{align*}
and again, the second term is negligible since it is a sum of conditionally independent random variables. Thus, we obtain
\begin{align*}
-(Nh^{2(1-\beta)})^{-1}\partial_\alpha^2 \mathbb{H}_N^\dagger (\theta_N)	
&= -\frac{1}{N}\sum_{n=1}^N\left(\frac{g_\beta(\epsilon_n(\theta_N))}{\epsilon_n(\theta_N)}\right)\left(\frac{\partial_\alpha a_{n-1}(\alpha_N)}{c_{n-1}(\gamma_N)}\right)^{\otimes 2}+o_p(1)
\nn\\
&= C_\alpha^\dagger(\beta)\Sigma_{T,\alpha}(\theta_0)+o_p(1)
\end{align*}
by Proposition \ref{lem:lln}. 

Similarly, we have 
\begin{align*}
(Nh^{(1-\beta)})^{-1}\partial_\alpha\partial_\gamma \mathbb{H}_N^\dagger (\theta_N)	
=&-\frac{2}{N}\sum_{n=1}^N\frac{\epsilon_n(\theta_N)}{V_n}\left(\frac{\partial_\alpha a_{n-1}(\alpha_N)}{c_{n-1}(\gamma_N)}\right)\otimes 
\left(\frac{\partial_\gamma c_{n-1}(\gamma_N)}{c_{n-1}(\gamma_N)}\right)\\
=&\frac{2}{N}\sum_{n=1}^Ng_\beta(\epsilon_N(\theta_N))\left(\frac{\partial_\alpha a_{n-1}(\alpha_N)}{c_{n-1}(\gamma_N)}\right)\otimes 
\left(\frac{\partial_\gamma c_{n-1}(\gamma_N)}{c_{n-1}(\gamma_N)}\right)\\
&-\frac{2}{N}\sum_{n=1}^N\left(g_\beta(\epsilon_N(\theta_N))+\frac{\epsilon_n(\theta_N)}{V_n}\right)\left(\frac{\partial_\alpha a_{n-1}(\alpha_N)}{c_{n-1}(\gamma_N)}\right)\otimes 
\left(\frac{\partial_\gamma c_{n-1}(\gamma_N)}{c_{n-1}(\gamma_N)}\right).
\end{align*}
The first term is $o_p(1)$ by Proposition \ref{lem:lln} together with the fact $\int g_\beta(x)\phi_\beta(x)\dif x=0$, 
and the second term converges to $0$ in $L^2$ since it is a sum of conditionally independent random variables. Then 
$(Nh^{(1-\beta)})^{-1}\partial_\alpha\partial_\gamma \mathbb{H}_N^\dagger (\theta_N)	\rightarrow 0$ in probability. 

In the same way, decomposing the sum into the main term and the sum of independent variables, we have
\begin{align*}
N^{-1}\partial_\gamma^2 \mathbb{H}_N^\dagger (\theta_N)	
=&-\frac{2}{N}\sum_{n=1}^N\frac{\epsilon_n(\theta_N)^2}{V_n}
\left(\frac{\partial_\gamma c_{n-1}(\gamma_N)}{c_{n-1}(\gamma_N)}\right)^{\otimes 2}\\
&+
\frac{1}{N}\sum_{n=1}^N\left(\frac{\epsilon_n(\theta_N)^2}{V_n}-1\right)
\left\{
\left(\frac{\partial_\gamma^2 c_{n-1}(\gamma_N)}{c_{n-1}(\gamma_N)}\right)
-\left(\frac{\partial_\gamma c_{n-1}(\gamma_N)}{c_{n-1}(\gamma_N)}\right)^{\otimes 2}\right\}\\
=&\frac{2}{N}\sum_{n=1}^N\epsilon_n(\theta_N)g_\beta(\epsilon_n(\theta_N))
\left(\frac{\partial_\gamma c_{n-1}(\gamma_N)}{c_{n-1}(\gamma_N)}\right)^{\otimes 2}\\
&+
\frac{1}{N}\sum_{n=1}^N\left(-\epsilon_n(\theta_N)g_\beta(\epsilon_n(\theta_N))-1\right)
\left\{
\left(\frac{\partial_\gamma^2 c_{n-1}(\gamma_N)}{c_{n-1}(\gamma_N)}\right)
-\left(\frac{\partial_\gamma c_{n-1}(\gamma_N)}{c_{n-1}(\gamma_N)}\right)^{\otimes 2}\right\}\\
&+o_p(1).
\end{align*}
The first term in the right-hand side converges to $-C_\gamma^\dagger(\beta)\Sigma_{T,\gamma}(\theta_0)$
and the second term is $o_p(1)$ by Lemma \ref{lem:lln} together with the fact that $\int xg(x)\phi_\beta(x)\dif x=-1$. 
Then, we obtain
\begin{align*}
-N^{-1}\partial_\gamma^2 \mathbb{H}_N^\dagger (\theta_N)		\rightarrow C_\gamma^\dagger(\beta)\Sigma_{T,\gamma}(\theta_0)
\end{align*}
in probability. Almost the same arguments give convergence of the thrice derivatives. We omit the detail. 
\end{proof}

\subsection{Proof of Theorem \ref{prop:mwg}}\label{sec:proof_mwg}

The proof of Theorem \ref{prop:mwg} is an application of Proposition \ref{prop:mtk}. Since $\Pi(u|s)=N_p(I(\theta_0)^{-1}s, I(\theta_0)^{-1})$ and $Q(u,\dif v|s)=N_p(u,\Sigma)$,  Assumption \ref{as:kernel} is easy to check. Ergodicity is also obvious by irreducibility. 
The Bernstein-von Mises theorem is already proved in Theorem \ref{hm:thm_bvM} which is a sufficient condition for (\ref{BvMMH}) in this case. 
We will complete the proof  by showing the convergence of the acceptance ratio $A_N$.

Let $\theta_N=\theta_0+D_N^{-1}u$ and $\theta_N^*=\theta_0+D_N^{-1}v$. 
By (\ref{eq:clan_fishermat}) and (\ref{eq:clan_remainder}) together with Taylor's expansion, we have the difference of the log quasi-posterior density satisfies 
\begin{align*}
	\xi_N &:=\{\mathbb{H}_N^\dagger(\theta^*_N)+\log\pi(\theta_N^*)\}-\{\mathbb{H}_N^\dagger(\theta_N)+\log\pi(\theta_N)\}\\
	&=\partial_\theta\mathbb{H}^\dagger_N(\theta_N^*-\theta_N)
	+\frac{1}{2}\partial_\theta^2\mathbb{H}_N^\dagger[(\theta_N^*-\theta)^{\otimes 2}]+o_p(1)\\
	&=\Delta^\dagger_N(\theta_N)[v-u]-\frac{1}{2}I^\dagger(\theta_0)[(v-u)^{\otimes 2}]+o_p(1).
\end{align*}
Now we rewrite the two terms in the right-hand side of the above equation. 
First, observe that difference of the score function at $\theta_N$ and $\theta_0$ is 
$$
\Delta_{N}(\theta_N)-\Delta_{N}(\theta_0)=D_N^{-1}\partial_\theta^2\mathbb{H}_N(\theta_0)D_n^{-1}[u]+o_p(1)=-I(\theta_0)[u]+o_p(1)
$$
by Taylor's expansion. Also, by $I^\dagger(\theta_0)=I^*(\theta_0)+I(\theta_0)$, we have the following identity among Fisher information matrices: 
$$
I(\theta_0)[u, v-u] +\frac{1}{2}I^\dagger(\theta_0)[(v-u)^{\otimes 2}]=I(\theta_0)[v^{\otimes 2}-u^{\otimes 2}]+\frac{1}{2}I^*(\theta_0)[(v-u)^{\otimes 2}]. 
$$
By these facts, we can rewrite $\xi_N$ by
\begin{align*}
\xi_N=&\Delta_N(\theta_N)[v-u]+\Delta(\theta_N)[v-u]-\frac{1}{2}I^\dagger(\theta_0)[(v-u)^{\otimes 2}]+o_p(1)\\
=&\Delta_N(\theta_N)[v-u]+\Delta(\theta_0)[v-u]-I(\theta_0)[u,v-u]-\frac{1}{2}I^\dagger(\theta_0)[(v-u)^{\otimes 2}]+o_p(1)\\
=&\eta_N(\Delta_{N}^*(\theta_N))+o_p(1)
\end{align*}
where 
$$
\eta_N(w)=\Delta_{N}(\theta_0)[v-u]+w[v-u]-I(\theta_0)[v^{\otimes 2}-u^{\otimes 2}]-\frac{1}{2}I^*(\theta_0)[(v-u)^{\otimes 2}]
$$
for $w\in\mathbb{R}^p$. By the expression of the difference of the log quasi-posterior density together with the function $\psi(x)=\min\{1,\exp(x)\}$, we have the expression of the acceptance probabilities: 
\begin{align*}
	&A_N(\theta_N,\theta^*_N|X^N)=\tilde{\mathbb{E}}\left[\psi(\xi_N)|\omega\right], A(u, v|\Delta_{N}(\theta_0), I(\theta_0), I^*(\theta_0))=\mathbb{E}\left[\psi\left(\eta_N(W)\right)|w\right]
\end{align*}
where $W\sim N(0, I^*(\theta_0))$. 
On the other hand,  we have
\begin{align*}
\left|\tilde{\mathbb{E}}[\psi(\xi_N)|\omega]-{\mathbb{E}}[\psi(\eta_N(W))|w]\right|
	\le& \left|\tilde{\mathbb{E}}[\psi(\xi_N)|\omega]-{\mathbb{E}}[\psi(\eta_N(\Delta_{N}^*(\theta_N))|w]\right|\\
	&+ \left|\tilde{\mathbb{E}}[\psi(\eta_N(\Delta_{N}^*(\theta_N))|\omega]-{\mathbb{E}}[\psi(\eta_N(W))|w]\right|\\
\le& o_p(1)+\|\mathcal{L}(\Delta^*_{N,T}(\theta_N)|w)-N(0, I^*(\theta_0))\|_{\mathrm{BL}}=o_p(1). 
\end{align*}
Hence the claim follows.

%%%%%

\end{document}